\documentclass[11pt]{amsart}
\usepackage{
amssymb
,amsthm
,amsmath
,amscd
,mathtools,
mathdots,
url
}
\usepackage[all]{xy}
\usepackage[top=35truemm,bottom=35truemm,left=30truemm,right=30truemm]{geometry}

\newcommand{\Z}{\mathbb{Z}}
\newcommand{\R}{\mathbb{R}}
\newcommand{\C}{\mathbb{C}}

\renewcommand{\AA}{\mathcal{A}}
\newcommand{\Sc}{\mathcal{S}}
\newcommand{\EE}{\mathcal{E}}
\newcommand{\FF}{\mathcal{F}}

\newcommand{\GL}{\mathrm{GL}}
\newcommand{\SL}{\mathrm{SL}}
\newcommand{\SO}{\mathrm{SO}}
\newcommand{\Sp}{\mathrm{Sp}}

\newcommand{\Irr}{\mathrm{Irr}}
\newcommand{\unit}{\mathrm{unit}}
\newcommand{\gp}{\mathrm{gp}}
\newcommand{\temp}{\mathrm{temp}}
\newcommand{\Ind}{\mathrm{Ind}}
\newcommand{\Jac}{\mathrm{Jac}}
\newcommand{\soc}{\mathrm{soc}}
\newcommand{\Cusp}{\mathrm{Cusp}}

\newcommand{\iif}{&\quad&\text{if }}
\newcommand{\other}{&\quad&\text{otherwise}}
\newcommand{\resp}{resp.~}
\renewcommand{\1}{\mathbf{1}}
\newcommand{\ep}{\varepsilon}

\newcommand{\pair}[1]{\left\langle #1 \right\rangle}
\newcommand{\half}[1]{\frac{#1}{2}}

\newtheorem{thm}{Theorem}[section]

\newtheorem{prop}[thm]{Proposition}
\newtheorem{cor}[thm]{Corollary}

\newtheorem{defi}[thm]{Definition}
\newtheorem{ex}[thm]{Example}
\newtheorem{alg}[thm]{Algorithm}

\title[On the socles of certain parabolically induced representations]
{On the socles of certain parabolically induced representations of $p$-adic classical groups}
\author{Hiraku Atobe}
\date{}
\address{
Department of Mathematics, Hokkaido University,
Kita 10, Nishi 8, Kita-Ku, Sapporo, Hokkaido, 060-0810, Japan 
}
\email{
atobe@math.sci.hokudai.ac.jp
}

\allowdisplaybreaks
\begin{document}
\maketitle

\begin{abstract}
In this paper, we consider representations of $p$-adic classical groups
parabolically induced from 
the products of shifted Speh representations 
and unitary representations of Arthur type of good parity. 
We describe 
how to compute the socles (the maximal semisimple subrepresentations) of these representations. 
As a consequence, we can determine whether these representations are reducible or not.
In particular, our results produce many unitary representations, 
which appear in the complementary series.
\end{abstract}

\tableofcontents

%\section{Introduction}
%\section{Introduction}
\section{Introduction}
One of the most important problems in representation theory is 
the classification of irreducible unitary representations of a given group. 
This is called the unitary dual problem.
In this paper, we consider $p$-adic reductive groups. 
Let $F$ be a non-archimedean local field of characteristic zero. 
The unitary duals of general linear groups $\GL_n(F)$ 
were explicitly described by Tadi{\'c} \cite{T-GL} in the 1980's. 
According to his description, 
the union of the unitary duals of all $\GL_n(F)$ has two important classes; 
the \emph{Speh representations} and their \emph{complementary series}. 
Indeed, any irreducible unitary representation of $\GL_n(F)$
is obtained as the irreducible parabolically induced representation
from a product of complementary series. 
\vskip 10pt

Now we consider classical groups. 
Let $G_n$ be a split special odd orthogonal group $\SO_{2n+1}(F)$ 
or a symplectic group $\Sp_{2n}(F)$ of rank $n$ over $F$.  
There are several works on the unitary dual of $G_n$. 
For example: 
\begin{itemize}
\item
The generic unitary dual of $G_n$ was explicitly described by Lapid--Mui{\'c}--Tadi{\'c} \cite{LMT}. 
\item
When $n \leq 3$, 
Tadi{\'c} \cite{T-rank3} computed the unitary dual of $G_n$ completely. 
\end{itemize}
However, the general case is still widely unknown.
\vskip 10pt

The notion of \emph{local $A$-packets} was introduced by Arthur \cite{Ar}. 
They are (multi-)sets over 
the set of equivalence classes of irreducible unitary representations of $G_n$. 
We say that an irreducible representation is \emph{of Arthur type}
if it belongs to some local $A$-packet. 
One can regard representations of Arthur type as analogues of Speh representations 
in the sense that both of them are local factors of square-integrable automorphic representations.
In fact, representations of Arthur type are expected to play an alternative role
to Speh representations  in the unitary dual problem (see \cite[Conjectures 8.2, 8.3, 8.5]{T-arthur}). 
\vskip 10pt

In this paper, we consider the parabolically induced representation of the form
\[
\Pi_s = u_\rho(a,b)|\cdot|^s \rtimes \pi_A
\]
where 
\begin{itemize}
\item
$u_\rho(a,b)$ is the (unitary) Speh representation with 
an irreducible self-dual supercuspidal representation $\rho$ of $\GL_d(F)$ 
and positive integers $a,b$
(see Section \ref{s.gl}); 
\item
$\pi_A$ is an irreducible representation of Arthur type (see Section \ref{Apacket}); 
\item
$s \in \R$.
\end{itemize}
It is known that: 
\begin{itemize}
\item
$\Pi_0$ decomposes into a multiplicity-free direct sum of irreducible unitary representations of Arthur type 
(\cite[Proposition 2.4.3]{Ar}). 
\item
If $s_0$ is the \emph{first reducibility point} for $\Pi_s$, i.e.,
the minimal non-negative real number such that $\Pi_{s_0}$ is reducible, 
then $\Pi_s$ is irreducible and unitary for any $0 \leq s < s_0$ (see \cite[Section 3 (b)]{T-ext}). 
In this case, $\Pi_s$ is called a \emph{complementary series representation}. 
\item
All irreducible constituents of $\Pi_{s_0}$ are also unitary (see \cite[Section 3 (c)]{T-ext}). 
\end{itemize}
In particular, the study of $\Pi_s$ would produce many irreducible unitary representations. 
\vskip 10pt

In several restricted cases, the reducibility and the composition series of $\Pi_s$ are already known. 
For example, 
several criteria for the irreducibility of $\Pi_s$ 
were given by 
\begin{itemize}
\item
Mui{\'c} \cite{Mu2} 
when $b = 1$ and $\pi_A$ is a discrete series representation; 

\item
Mati{\'c} \cite{Ma1} 
when $a = 1$ and $\pi_A$ is a discrete series representation;

\item
Lapid--Tadi{\'c} \cite[Theorems 1.1, 1.2]{LT} 
when $(a,b)$ is arbitrary and $\pi_A$ is supercuspidal. 
\end{itemize}
All irreducible constituents of $\Pi_s$ were obtained by
\begin{itemize}
\item
Mui{\'c} \cite{Mu1}
when $b = 1$ and $\pi_A$ is a strongly positive discrete series representation;

\item
Mati{\'c} \cite{Ma2}
when $a = 1$, $b+2s \in 2\Z$ and $\pi_A$ is a discrete series representation;

\item
Bo{\v s}njak \cite{B}
when $\pi_A$ is supercuspidal and $s \geq (a+b-1)/2$.
\end{itemize}
Finally, the semisimplification of $\Pi_s$ for $s \in \{0, 1/2\}$ 
seems to be already given by M{\oe}glin 
(watch the video of her talk \cite{M-video}). 
\vskip 10pt

In this paper, we describe the \emph{socle} $\soc(\Pi_s)$ of $\Pi_s$, i.e., 
the maximal semisimple subrepresentation of $\Pi_s$. 
The main theorem is as follows. 

\begin{thm}\label{main}
Let $\Pi_s = u_\rho(a,b)|\cdot|^s \rtimes \pi_A$ be as above. 
Then we can describe the socle $\soc(\Pi_s)$ algorithmically in terms of the Langlands classification. 
Moreover, the following hold. 
\begin{enumerate}
\item
If $s > (a-1)/2$ or $s < -(b-1)/2$, or if $s \not\in (1/2)\Z$, 
then the socle $\soc(\Pi_s)$ is irreducible. 
(See Proposition \ref{s>>0}.) 

\item
If $s \in (1/2)\Z$ and $0 < s \leq (a-1)/2$ (\resp $-(b-1)/2 \leq s < 0$), 
then any irreducible subrepresentation of $\Pi_s$ is 
of the form $\pi' = \soc(u_\rho(2s,b)|\cdot|^{a/2} \rtimes \pi_A')$ 
(\resp $\pi' = \soc(u_\rho(a,-2s)|\cdot|^{-b/2} \rtimes \pi_A')$)
for some irreducible summand $\pi_A'$ of $u_\rho(a-2s,b) \rtimes \pi_A$ 
(\resp $u_\rho(a,b+2s) \rtimes \pi_A$). 
Moreover, for such an irreducible summand $\pi_A'$, 
one can determine whether $\pi'$ is a subrepresentation of $\Pi_s$ or not. 
(See Propositions \ref{middle} and \ref{image}.) 

\item
If $s=0$, then the decomposition of $\Pi_0$ is explicitly given in terms of extended multi-segments. 
(See Theorem \ref{s=0}.) 
\end{enumerate}
\end{thm}

We have several consequences. 
\begin{cor}[Corollary \ref{free}]
Let $\Pi_s = u_\rho(a,b)|\cdot|^s \rtimes \pi_A$ be as above. 
Then any irreducible subrepresentation of $\Pi_s$ appears in 
the composition series of $\Pi_s$ with multiplicity one. 
In particular, $\soc(\Pi_s)$ is multiplicity-free. 
\end{cor}

\begin{cor}[Corollary \ref{irred}]
Let $\Pi_s = u_\rho(a,b)|\cdot|^s \rtimes \pi_A$ be as above. 
Then $\Pi_s$ is irreducible if and only if all of the following conditions hold: 
\begin{itemize}
\item
$\soc(\Pi_s)$ is irreducible; 
\item
$\soc(\Pi_{-s})$ is irreducible; 
\item
$\soc(\Pi_s) \cong \soc(\Pi_{-s})$.
\end{itemize}
\end{cor}
In particular, one can compute the first reducibility point $s_0$ for $\Pi_s$ (Corollary \ref{FRP}). 
\vskip 10pt

This paper is organized as follows.
In Section \ref{s.langlands}, we review the Langlands classification of classical groups. 
In Section \ref{s.non-unitary}, we prove Theorem \ref{main} (1) and (2). 
To do this, we refine the theory of derivatives used in \cite{AM}. 
Theorem \ref{main} (3) is proven in Section \ref{s.unitary}
after reviewing the theory of extended multi-segments established in the previous paper \cite{At}.
Finally, in Section \ref{s.irred}, we obtain several consequences about the irreducibility of $\Pi_s$, 
and we give some examples. 
To describe $\soc(\Pi_s)$ explicitly, we need to compute several derivatives. 
In Appendix \ref{appA}, we give an algorithm for the computations of certain derivatives, 
which are not obtained in \cite{AM}. 
\vskip 10pt

\noindent
{\bf Acknowledgement.}
The author was supported by JSPS KAKENHI Grant Number 19K14494. 
He is very grateful to the referee for the careful readings and the helpful comments.
\vskip 10pt

\noindent
\textbf{Notation.}
Let $F$ be a non-archimedean local field of characteristic zero.
The normalized absolute value is denoted by $|\cdot|$, 
which is also regarded as a character of $\GL_d(F)$ via composing with the determinant map. 
\par

Let $G_n$ be a split special odd orthogonal group $\SO_{2n+1}(F)$ or a symplectic group $\Sp_{2n}(F)$ 
of rank $n$ over $F$. 
For a smooth representation $\Pi$ of $G_n$ or $\GL_n(F)$ of finite length, 
we write $[\Pi]$ for its semisimplification. 
Similarly, we denote by $\soc(\Pi)$ the \emph{socle} of $\Pi$, 
i.e., the maximal semisimple subrepresentation of $\Pi$.
The set of equivalence classes of irreducible smooth representations of a group $G$ is denoted by $\Irr(G)$. 
\par

We will often extend the set theoretical language to multi-sets. 
Namely, we write a multi-set as $\{x, \dots, x, y, \dots, y, \ldots \}$.
When we use a multi-set, we will mention it.

%\section{Preliminary}
%\section{Preliminary}
\section{Langlands classification}\label{s.langlands}
In this section, we review the Langlands classification of classical groups. 

%\subsection{General linear groups}
\subsection{General linear groups}\label{s.gl}
First, we recall some notations for representations of $\GL_n(F)$. 
Let $P$ be a standard parabolic subgroup of $\GL_n(F)$ 
with Levi subgroup $M \cong \GL_{n_1}(F) \times \dots \times \GL_{n_r}(F)$. 
For representations $\tau_1, \dots, \tau_r$ of $\GL_{n_1}(F), \dots, \GL_{n_r}(F)$, respectively,
we denote by 
\[
\tau_1 \times \dots \times \tau_r 
\coloneqq \Ind_P^{\GL_n(F)}(\tau_1 \boxtimes \dots \boxtimes \tau_r)
\]
the normalized parabolically induced representation. 
\par

Let $\Cusp_\unit(\GL_d(F))$ be the set of equivalence classes of 
irreducible unitary supercuspidal representations of $\GL_d(F)$, 
and $\Cusp^\bot(\GL_d(F))$ be the subset consisting of self-dual elements.
\par

A \emph{segment} $[x,y]_\rho$ is a set of supercuspidal representations of the form 
\[
[x,y]_\rho \coloneqq \{ \rho|\cdot|^x, \rho|\cdot|^{x-1}, \dots, \rho|\cdot|^y\}, 
\]
where $\rho \in \Cusp_\unit(\GL_d(F))$ and $x,y \in \R$ such that $x-y \in \Z$ and $x \geq y$.
For a segment $[x,y]_\rho$, define the \emph{Steinberg representation} $\Delta_\rho[x,y]$ 
as a unique irreducible subrepresentation of 
\[
\rho|\cdot|^x \times \dots \times \rho|\cdot|^{y}.
\]
This is an essentially discrete series representation of $\GL_{d(x-y+1)}(F)$. 
Similarly, we define $Z_\rho[y,x]$ 
as a unique irreducible quotient of the same induced representation.
By convention, we set $\Delta_\rho[x,x+1]$ and $Z_\rho[x+1,x]$ 
to be the trivial representation of the trivial group $\GL_0(F)$.
\par

The Langlands classification for $\GL_n(F)$ says that 
every $\tau \in \Irr(\GL_n(F))$ is a unique irreducible subrepresentation 
of $\Delta_{\rho_1}[x_1,y_1] \times \dots \times \Delta_{\rho_r}[x_r,y_r]$, 
where $\rho_i \in \Cusp_\unit(\GL_{d_i}(F))$ for $i = 1, \dots, r$ 
such that $x_1+y_1 \leq \dots \leq x_r+y_r$. 
In this case, we write
\[
\tau = L(\Delta_{\rho_1}[x_1,y_1], \dots, \Delta_{\rho_r}[x_r,y_r]).
\]
\par

When $(x_{i,j})_{1 \leq i \leq t, 1 \leq j \leq d}$ satisfies that $x_{i,j} = x_{1,1} - i + j$, 
the irreducible representation 
$L(\Delta_{\rho}[x_{1,1},x_{t,1}], \dots, \Delta_{\rho}[x_{1,d},x_{t,d}])$
is called a \emph{(shifted) Speh representation} and is denoted by
\[
\begin{pmatrix}
x_{1,1} & \ldots & x_{1,d} \\
\vdots & \ddots & \vdots \\
x_{t,1} & \ldots & x_{t,d}
\end{pmatrix}_\rho
\coloneqq L(\Delta_{\rho}[x_{1,1},x_{t,1}], \dots, \Delta_{\rho}[x_{1,d},x_{t,d}]).
\] 
Note that it is isomorphic to the unique irreducible subrepresentation of 
$Z_\rho[x_{1,1}, x_{1,d}] \times \dots \times Z_\rho[x_{t,1}, x_{t,d}]$.
Especially, for positive integers $a$ and $b$, 
set 
\[
u_\rho(a,b) = \begin{pmatrix}
\half{a-b} & \ldots & \half{a+b}-1 \\
\vdots & \ddots & \vdots \\
-\half{a+b}+1 & \ldots & -\half{a-b}
\end{pmatrix}_\rho. 
\]
It is an irreducible unitary representation.
We often set $A \coloneqq (a+b)/2-1$ and $B \coloneqq (a-b)/2$.

%\subsection{Classical groups}
\subsection{Classical groups}
Next we recall some notations for representations of the classical group $G_n$. 
Fix an $F$-rational Borel subgroup of $G_n$. 
Let $P$ be a standard parabolic subgroup of $G_n$ with Levi subgroup 
$M \cong \GL_{n_1}(F) \times \dots \times \GL_{n_r}(F) \times G_{n_0}$. 
For representations $\tau_1, \dots, \tau_r$ and $\pi_0$ 
of $\GL_{n_1}(F), \dots, \GL_{n_r}(F)$ and of $G_{n_0}$, respectively, 
we denote by
\[
\tau_1 \times \dots \times \tau_r \rtimes \pi_0 
\coloneqq
\Ind_P^{G_n}(\tau_1 \boxtimes \dots \boxtimes \tau_r \boxtimes \pi_0)
\]
the normalized parabolically induced representation.
\par

The Langlands classification for $G_n$ says that 
every $\pi \in \Irr(G_n)$ is a unique irreducible subrepresentation 
of $\Delta_{\rho_1}[x_1,y_1] \times \dots \times \Delta_{\rho_r}[x_r,y_r] \rtimes \pi_0$, 
where 
\begin{itemize}
\item
$\rho_i \in \Cusp_\unit(\GL_{d_i}(F))$ for $i = 1,\dots,r$; 
\item
$x_1+y_1 \leq \dots \leq x_r+y_r < 0$; 
\item
$\pi_0$ is an irreducible tempered representation of $G_{n_0}$.
\end{itemize}
In this case, we write
\[
\pi = L(\Delta_{\rho_1}[x_1,y_1], \dots, \Delta_{\rho_r}[x_r,y_r]; \pi_0),
\]
and call $(\Delta_{\rho_1}[x_1,y_1], \dots, \Delta_{\rho_r}[x_r,y_r]; \pi_0)$
the \emph{Langlands data} for $\pi$. 
\par

We say that $\pi \in \Irr(G_n)$ is \emph{of Arthur type}
if $\pi$ belongs to an $A$-packet associated to some $A$-parameter. 
For the notion of $A$-parameters and properties of representations of Arthur type, 
see Sections \ref{Apar} and \ref{Apacket} below. 
As basic properties, the following are known:
Let $\pi \in \Irr(G_n)$ be of Arthur type. 
Then $\pi$ is a unitary representation.
In particular, for $\rho \in \Cusp^\bot(\GL_d(F))$, 
the parabolically induced representation $u_\rho(a,b) \rtimes \pi$ is also unitary 
so it is semisimple. 
By \cite[Proposition 2.4.3]{Ar} (see Proposition \ref{lem223}), 
we know that $u_\rho(a,b) \rtimes \pi$ is multiplicity-free.

%\section{Non-unitary inductions}
%\section{Non-unitary inductions}
\section{Non-unitary inductions}\label{s.non-unitary}
Through this section, we fix 
\begin{itemize}
\item
$\rho \in \Cusp^\bot(\GL_d(F))$; 
\item
$\pi \in \Irr(G_n)$ of Arthur type; and
\item
positive integers $a$ and $b$. 
\end{itemize}
The purpose of this paper is to explain 
how to describe the socle of parabolically induced representation 
$u_\rho(a,b)|\cdot|^s \rtimes \pi$
for $s \in \R$.
In this section, we do it for $s \not= 0$.

%\subsection{Theory of derivatives}
\subsection{Theory of derivatives}\label{s.der}
In this subsection, we introduce the notion of \emph{derivatives}, 
which is the main terminology.
\par

For a smooth representation $\pi$ of $G_n$ of finite length, 
denote by $\Jac_{P}(\pi)$ its Jacquet module 
along a standard parabolic subgroup $P$. 
Let $P_d$ be the standard parabolic subgroup 
with Levi subgroup isomorphic to $\GL_d(F) \times G_{n-d}$. 
For $x \in \R$,
the \emph{$\rho|\cdot|^x$-derivative} $D_{\rho|\cdot|^x}(\pi)$ 
is a semisimple representation of $G_{n-d}$ satisfying that 
\[
[\Jac_{P_{d}}(\pi)] = \rho|\cdot|^x \boxtimes D_{\rho|\cdot|^x}(\pi) + \sum_{i} \tau_i \boxtimes \pi_i, 
\]
where $\tau_i \boxtimes \pi_i$ is an irreducible representation of $\GL_{d}(F) \times G_{n-d}$ 
such that $\tau_i \not \cong \rho|\cdot|^x$. 
We say that $\pi$ is \emph{$\rho|\cdot|^x$-reduced} if $D_{\rho|\cdot|^x}(\pi) = 0$. 
For a segment $[x,y]_\rho$, we set 
\begin{align*}
D_{\rho|\cdot|^{x}, \dots, \rho|\cdot|^{y}}(\pi) 
&= D_{\rho|\cdot|^{y}} \circ \dots \circ D_{\rho|\cdot|^{x}}(\pi), \\
D_{\rho|\cdot|^{y}, \dots, \rho|\cdot|^{x}}(\pi) 
&= D_{\rho|\cdot|^{x}} \circ \dots \circ D_{\rho|\cdot|^{y}}(\pi). 
\end{align*}
Hence, with a suitable parabolic subgroup $P$, we have 
\begin{align*}
[\Jac_{P}(\pi)] 
&= \rho|\cdot|^x \boxtimes \dots \boxtimes \rho|\cdot|^y 
\boxtimes D_{\rho|\cdot|^{x}, \dots, \rho|\cdot|^{y}}(\pi)
+ \text{(others)}, \\
[\Jac_{P}(\pi)] 
&= \rho|\cdot|^y \boxtimes \dots \boxtimes \rho|\cdot|^x 
\boxtimes D_{\rho|\cdot|^{y}, \dots, \rho|\cdot|^{x}}(\pi)
+ \text{(others)}. 
\end{align*}
\par

We also set $D_{\rho|\cdot|^x}^{(0)}(\pi) = \pi$ and 
\[
D_{\rho|\cdot|^x}^{(k)}(\pi) = \frac{1}{k}D_{\rho|\cdot|^x} \circ D_{\rho|\cdot|^x}^{(k-1)}(\pi)
= \frac{1}{k!} \underbrace{D_{\rho|\cdot|^x} \circ \dots \circ D_{\rho|\cdot|^x}}_k (\pi)
\]
for $k > 0$.
It satisfies that 
\[
[\Jac_{P_{dk}}(\pi)] 
= (\rho|\cdot|^x)^k \boxtimes D_{\rho|\cdot|^x}^{(k)}(\pi)
+ \text{(others)},
\]
where $(\rho|\cdot|^x)^k = \rho|\cdot|^x \times \dots \times \rho|\cdot|^x$ ($k$ times).
When $D_{\rho|\cdot|^x}^{(k)}(\pi) \not= 0$ but $D_{\rho|\cdot|^x}^{(k+1)}(\pi) = 0$, 
we call $D_{\rho|\cdot|^x}^{(k)}(\pi)$ the \emph{highest $\rho|\cdot|^x$-derivative} of $\pi$, 
and set $D_{\rho|\cdot|^x}^{\max}(\pi) \coloneqq D_{\rho|\cdot|^x}^{(k)}(\pi)$.
Note that if $|x-x'| \not= 1$, then 
$D_{\rho|\cdot|^x}^{(k)} \circ D_{\rho|\cdot|^{x'}}^{(k')}(\pi) 
= D_{\rho|\cdot|^{x'}}^{(k')} \circ D_{\rho|\cdot|^{x}}^{(k)}(\pi)$
(see \cite[Lemma 5.6]{X1}).
\par

Although it is difficult to describe $D_{\rho|\cdot|^x}(\pi)$, 
one can compute $D_{\rho|\cdot|^x}^{\max}(\pi)$ when $x \not= 0$. 
\begin{thm}[{\cite[Lemma 3.1.3]{J-temp}}, {\cite[Propositions 3.3, 6.1, Theorem 7.1]{AM}}]\label{der}
Suppose that $x \not= 0$ so that $\rho|\cdot|^x$ is not self-dual. 
Let $\pi$ be an irreducible representation of $G_n$. 
Then the highest $\rho|\cdot|^x$-derivative $D_{\rho|\cdot|^x}^{\max}(\pi)$ is irreducible. 
Moreover, 
the Langlands data for $D_{\rho|\cdot|^x}^{\max}(\pi)$
can be described from those for $\pi$ explicitly, and vice versa. 
\end{thm}

When $x=0$, the $\rho$-derivative is more difficult. 
As alternatives of $\rho$-derivative, following \cite{AM}, 
we introduce other two derivatives. 
We define the \emph{$\Delta_\rho[0,-1]$-derivative $D_{\Delta_\rho[0,-1]}^{(k)}(\pi)$}
and the \emph{$Z_\rho[0,1]$-derivative $D_{Z_\rho[0,1]}^{(k)}(\pi)$} 
as semisimple representations of $G_{n-2dk}$ satisfying 
\[
[\Jac_{P_{2dk}}(\pi)] = \Delta_\rho[0,-1]^k \boxtimes D_{\Delta_\rho[0,-1]}^{(k)}(\pi)
+ Z_\rho[0,1]^k \boxtimes D_{Z_\rho[0,1]}^{(k)}(\pi)
+ \sum_{i} \tau_i \boxtimes \pi_i, 
\]
where $\tau_i \boxtimes \pi_i$ is an irreducible representation of $\GL_{2dk}(F) \times G_{n-2dk}$ 
such that $\tau_i \not \cong \Delta_{\rho}[0,-1]^k, Z_\rho[0,1]^k$.
We set 
$D_{\Delta_\rho[0,-1]}^{\max}(\pi) = D_{\Delta_\rho[0,-1]}^{(k)}(\pi)$
(\resp $D_{Z_\rho[0,1]}^{\max}(\pi) = D_{Z_\rho[0,1]}^{(k)}(\pi)$)
when $D_{\Delta_\rho[0,-1]}^{(k)}(\pi) \not= 0$ but $D_{\Delta_\rho[0,-1]}^{(k+1)}(\pi) = 0$
(\resp $D_{Z_\rho[0,1]}^{(k)}(\pi) \not= 0$ but $D_{Z_\rho[0,1]}^{(k+1)}(\pi) = 0$). 

\begin{thm}[{\cite[Proposition 3.7]{AM}}, Section \ref{appA}]\label{der2}
Let $\pi$ be an irreducible representation of $G_n$. 
Suppose that $\pi$ is $\rho|\cdot|^{-1}$-reduced (\resp $\rho|\cdot|^1$-reduced). 
Then the same assertions in Theorem \ref{der} hold when $\rho|\cdot|^x$ is replaced with 
$\Delta_\rho[0,-1]$ (\resp $Z_\rho[0,1]$). 
\end{thm}

To deal with these derivatives uniformly,
we introduce the following notation. 
\begin{defi}\label{def.der}
Fix a segment $[x,y]_\rho$. 
Let $\pi$ be a smooth representation of $G_n$ of finite length. 

\begin{enumerate}
\item
If $[x,y]_\rho$ does not contain $\rho$, then we set
\begin{align*}
D^{\max}_{\rho|\cdot|^{x}, \dots, \rho|\cdot|^{y}}(\pi) 
&\coloneqq 
D^{\max}_{\rho|\cdot|^y} \circ \dots \circ D^{\max}_{\rho|\cdot|^x}(\pi), \\
D^{\max}_{\rho|\cdot|^{y}, \dots, \rho|\cdot|^{x}}(\pi) 
&\coloneqq 
D^{\max}_{\rho|\cdot|^x} \circ \dots \circ D^{\max}_{\rho|\cdot|^y}(\pi).
\end{align*}

\item
Suppose that $[x,y]_\rho$ contains $\rho$, and that $y < 0$. 
Then we set
\[
D^{\max}_{\rho|\cdot|^{x}, \dots, \rho|\cdot|^{y}}(\pi)
\coloneqq
D^{\max}_{\rho|\cdot|^y} \circ \dots \circ D^{\max}_{\rho|\cdot|^{-2}}
\circ \left(D^{\max}_{\Delta_\rho[0,-1]} \circ D_{\rho|\cdot|^{-1}}^{\max}\right) \circ 
D^{\max}_{\rho|\cdot|^1} \circ \dots \circ D^{\max}_{\rho|\cdot|^x}(\pi).
\]
Moreover, if 
\[
D^{\max}_{\rho|\cdot|^{x}, \dots, \rho|\cdot|^{y}}(\pi) 
= 
D^{(k_y)}_{\rho|\cdot|^y} \circ \dots \circ D^{(k_{-2})}_{\rho|\cdot|^{-2}}
\circ \left( D^{(k_0)}_{\Delta_\rho[0,-1]} \circ D_{\rho|\cdot|^{-1}}^{(k_{-1})} \right) \circ 
D^{(k_1)}_{\rho|\cdot|^1} \circ \dots \circ D^{(k_x)}_{\rho|\cdot|^x}(\pi),
\]
we formally write 
\[
D^{\max}_{\rho|\cdot|^{x}, \dots, \rho|\cdot|^{y}}(\pi) 
= 
D^{(k_y)}_{\rho|\cdot|^y} \circ \dots \circ D^{(k_x)}_{\rho|\cdot|^x}(\pi). 
\]

\item
Suppose that $[x,y]_\rho$ contains $\rho$, and that $x > 0$. 
Then we set
\[
D^{\max}_{\rho|\cdot|^{y}, \dots, \rho|\cdot|^{x}}(\pi) 
\coloneqq
D^{\max}_{\rho|\cdot|^x} \circ \dots \circ D^{\max}_{\rho|\cdot|^{2}}
\circ \left( D^{\max}_{Z_\rho[0,1]} \circ D_{\rho|\cdot|^{1}}^{\max} \right) \circ 
D^{\max}_{\rho|\cdot|^{-1}} \circ \dots \circ D^{\max}_{\rho|\cdot|^y}(\pi).
\]
Moreover, if 
\[
D^{\max}_{\rho|\cdot|^{y}, \dots, \rho|\cdot|^{x}}(\pi) 
= 
D^{(k_x)}_{\rho|\cdot|^x} \circ \dots \circ D^{(k_2)}_{\rho|\cdot|^{2}}
\circ \left( D^{(k_0)}_{Z_\rho[0,1]} \circ D_{\rho|\cdot|^{1}}^{(k_1)} \right) \circ 
D^{(k_{-1})}_{\rho|\cdot|^{-1}} \circ \dots \circ D^{(k_y)}_{\rho|\cdot|^y}(\pi),
\]
we formally write
\[
D^{\max}_{\rho|\cdot|^{y}, \dots, \rho|\cdot|^{x}}(\pi) 
= 
D^{(k_x)}_{\rho|\cdot|^x} \circ \dots \circ D^{(k_y)}_{\rho|\cdot|^y}(\pi).
\]
\end{enumerate}
\end{defi}

By Theorems \ref{der} and \ref{der2}, if $\pi$ is irreducible, then 
$D^{\max}_{\rho|\cdot|^{x}, \dots, \rho|\cdot|^{y}}(\pi)$ is also irreducible 
whenever it is defined.
Moreover, 
the Langlands data for $D^{\max}_{\rho|\cdot|^{x}, \dots, \rho|\cdot|^{y}}(\pi)$
can be described from those for $\pi$ explicitly, and vice versa. 
Similar statements also hold for $D^{\max}_{\rho|\cdot|^{y}, \dots, \rho|\cdot|^{x}}(\pi)$.

%\subsection{The case where $|s| \gg 0$}
\subsection{The case where $|s| \gg 0$}
In this subsection, 
we study $\soc(u_\rho(a,b)|\cdot|^s \rtimes \pi)$ when $|s| \gg 0$ or $s \not\in (1/2)\Z$. 

\begin{prop}\label{s>>0}
Assume one of the following: 
\begin{itemize}
\item
$s > (a-1)/2$; 
\item
$s < -(b-1)/2$; or
\item
$s \not\in (1/2)\Z$. 
\end{itemize} 
Then for any $\pi \in \Irr(G_n)$, 
the socle $\soc(u_\rho(a,b)|\cdot|^s \rtimes \pi)$ is irreducible. 
Moreover, 
it appears in the semisimplification $[u_\rho(a,b)|\cdot|^s \rtimes \pi]$ with multiplicity one. 
\end{prop}
\begin{proof}
Let $\pi'$ be an irreducible subrepresentation of $u_\rho(a,b)|\cdot|^s \rtimes \pi$.
Note that 
\[
u_\rho(a,b)|\cdot|^s = 
\begin{pmatrix}
B+s & \ldots & A+s \\
\vdots & \ddots & \vdots \\
-A+s & \ldots & -B+s
\end{pmatrix}_\rho
\]
with $A = (a+b)/2-1$ and $B = (a-b)/2$. 
Hence the condition $s > (a-1)/2$ (\resp $s < -(b-1)/2$) 
implies that $-B+s > 0$ and $-(-B+s) < -A+s$ (\resp $-B+s < 0$ and $-(-B+s) > A+s$).
Therefore, if $s > (a-1)/2$, then 
\begin{align*}
&D^{\max}_{\rho|\cdot|^{-A+s}, \dots, \rho|\cdot|^{-B+s}} 
\circ \dots \circ D^{\max}_{\rho|\cdot|^{B+s}, \dots, \rho|\cdot|^{A+s}}(\pi')
\\&=
D^{\max}_{\rho|\cdot|^{-A+s}, \dots, \rho|\cdot|^{-B+s}} 
\circ \dots \circ D^{\max}_{\rho|\cdot|^{B+s}, \dots, \rho|\cdot|^{A+s}}(\pi),
\end{align*}
whereas if $s < -(b-1)/2$, then
\begin{align*}
&D^{\max}_{\rho|\cdot|^{A+s}, \dots, \rho|\cdot|^{-B+s}} 
\circ \dots \circ D^{\max}_{\rho|\cdot|^{B+s}, \dots, \rho|\cdot|^{-A+s}}(\pi')
\\&=
D^{\max}_{\rho|\cdot|^{A+s}, \dots, \rho|\cdot|^{-B+s}} 
\circ \dots \circ D^{\max}_{\rho|\cdot|^{B+s}, \dots, \rho|\cdot|^{-A+s}}(\pi).
\end{align*}
These equations determine $\pi'$ uniquely. 
Moreover, it follows that $\pi'$ appears in $[u_\rho(a,b)|\cdot|^s \rtimes \pi]$ with multiplicity one. 
When $s \not\in (1/2)\Z$, the same argument works.
\end{proof}

\begin{ex}\label{ex.35}
Let $\pi$ be an irreducible representation of $G_n$. 
Then $Z_\rho[-1,2] \rtimes \pi$ has a unique irreducible subrepresentation. 
If 
\[
D^{\max}_{\rho|\cdot|^2} 
\circ \left( D^{\max}_{Z_\rho[0,1]} \circ D^{\max}_{\rho|\cdot|^1} \right)
\circ D^{\max}_{\rho|\cdot|^{-1}}(\pi)
= 
D^{(k_{2})}_{\rho|\cdot|^2} 
\circ \left( D^{(k_{0})}_{Z_\rho[0,1]} \circ D^{(k_1)}_{\rho|\cdot|^1} \right)
\circ D^{(k_{-1})}_{\rho|\cdot|^{-1}}(\pi), 
\]
then $\pi' = \soc(Z_\rho[-1,2] \rtimes \pi)$ is uniquely determined by 
\begin{align*}
D^{\max}_{\rho|\cdot|^2} 
\circ \left( D^{\max}_{Z_\rho[0,1]} \circ D^{\max}_{\rho|\cdot|^1} \right)
\circ D^{\max}_{\rho|\cdot|^{-1}}(\pi')
&= 
D^{(k_{2}+1)}_{\rho|\cdot|^2} 
\circ \left( D^{(k_{0}+1)}_{Z_\rho[0,1]} \circ D^{(k_1)}_{\rho|\cdot|^1} \right)
\circ D^{(k_{-1}+1)}_{\rho|\cdot|^{-1}}(\pi')
\\&=
D^{(k_{2})}_{\rho|\cdot|^2} 
\circ \left( D^{(k_{0})}_{Z_\rho[0,1]} \circ D^{(k_1)}_{\rho|\cdot|^1} \right)
\circ D^{(k_{-1})}_{\rho|\cdot|^{-1}}(\pi).
\end{align*}
Note that the last equation also determines $\pi$ uniquely from $\pi'$.
\end{ex}

%\subsection{The middle case}
\subsection{The middle case}
Next, we consider the case where $0 < s \leq (a-1)/2$ or $-(b-1)/2 \leq s < 0$.
In this case, by Propositions \ref{middle} and \ref{image} below, 
we reduce the problem to the case where $s = 0$. 

\begin{prop}\label{middle}
Suppose that $s \in (1/2)\Z$. 
Let $\pi \in \Irr(G_n)$ be of Arthur type, 
and $\pi'$ be an irreducible subrepresentation of $u_\rho(a,b)|\cdot|^s \rtimes \pi$. 

\begin{enumerate}
\item
If $0 < s \leq (a-1)/2$, then 
there exists a unique irreducible summand $\sigma$ of $u_\rho(a-2s, b) \rtimes \pi$ 
such that $\pi' = \soc(u_\rho(2s, b)|\cdot|^{\half{a}} \rtimes \sigma)$. 

\item
If $-(b-1)/2 \leq s < 0$, 
then 
there exists a unique irreducible summand $\sigma$ of $u_\rho(a, b+2s) \rtimes \pi$ 
such that $\pi' = \soc(u_\rho(a, -2s)|\cdot|^{-\half{b}} \rtimes \sigma)$. 
\end{enumerate}
Moreover, in the both cases, 
$\pi'$ appears in the semisimplification $[u_\rho(a,b)|\cdot|^s \rtimes \pi]$ with multiplicity one. 
\end{prop}
\begin{proof}
We only prove the case where $0 < s \leq (a-1)/2$. 
The other case is proven similarly. 
\par

When $0 < s \leq (a-1)/2$, 
since $u_\rho(a,b)|\cdot|^s \hookrightarrow u_\rho(2s, b)|\cdot|^{\half{a}} \times u_\rho(a-2s, b)$, 
we have $\pi' \hookrightarrow u_\rho(2s, b)|\cdot|^{\half{a}} \rtimes \sigma$ for 
some irreducible summand $\sigma$ of $u_\rho(a-2s, b) \rtimes \pi$. 
Since $a/2 > (2s-1)/2$, by Proposition \ref{s>>0}, 
the socle $\soc(u_\rho(2s, b)|\cdot|^{\half{a}} \rtimes \sigma)$ is irreducible.
Hence $\pi' = \soc(u_\rho(2s, b)|\cdot|^{\half{a}} \rtimes \sigma)$.
By this equation, $\sigma$ is uniquely determined by $\pi'$ using derivatives
(see the proof of Proposition \ref{s>>0} and Example \ref{ex.35}). 
Moreover, since $u_\rho(a-2s, b) \rtimes \pi$ is 
a multiplicity-free sum of irreducible representations (\cite[Proposition 2.4.3]{Ar}), 
and since $\pi'$ appears in $[u_\rho(2s, b)|\cdot|^{\half{a}} \rtimes \sigma]$ with multiplicity one, 
we conclude that $\pi'$ appears in $[u_\rho(a,b)|\cdot|^s \rtimes \pi]$ with multiplicity one. 
\end{proof}

By Proposition \ref{middle}, 
when $0 < s \leq (a-1)/2$, we obtained a well-defined injective map
\begin{align*}
\left\{ \pi' \hookrightarrow u_\rho(a,b)|\cdot|^s \rtimes \pi \right\}
\rightarrow 
\left\{ \sigma \hookrightarrow u_\rho(a-2s, b) \rtimes \pi \right\}
\end{align*}
characterized by $\pi' = \soc(u_\rho(2s, b)|\cdot|^{\half{a}} \rtimes \sigma)$. 
Similarly, when $-(b-1)/2 \leq s < 0$, we obtained a well-defined injective map
\begin{align*}
\left\{ \pi' \hookrightarrow u_\rho(a,b)|\cdot|^s \rtimes \pi \right\}
\rightarrow 
\left\{ \sigma \hookrightarrow u_\rho(a, b+2s) \rtimes \pi \right\}
\end{align*}
characterized by $\pi' = \soc(u_\rho(a, -2s)|\cdot|^{-\half{b}} \rtimes \sigma)$. 
We detemine the images of these maps. 

\begin{prop}\label{image}
Suppose that $s \in (1/2)\Z$, 
and that $\pi \in \Irr(G_n)$ is of Arthur type.
\begin{enumerate}
\item
When $0 < s \leq (a-1)/2$, 
for an irreducible summand $\sigma$ of $u_\rho(a-2s, b) \rtimes \pi$, 
the following are equivalent. 
\begin{enumerate}
\item
$\soc(u_\rho(2s, b)|\cdot|^{\half{a}} \rtimes \sigma)$
is a subrepresentation of $u_\rho(a,b)|\cdot|^s \rtimes \pi$; 
\item
if 
\begin{align*}
&D^{\max}_{\rho|\cdot|^{B-s+1}, \dots, \rho|\cdot|^{A-s+1}}
\circ \dots \circ
D^{\max}_{\rho|\cdot|^{B+s}, \dots, \rho|\cdot|^{A+s}}(\pi)
\\&= 
\left(D^{(k_{2s,b})}_{\rho|\cdot|^{A-s+1}} \circ \dots \circ D^{(k_{2s,1})}_{\rho|\cdot|^{B-s+1}}\right)
\circ \dots \circ
\left(D^{(k_{1,b})}_{\rho|\cdot|^{A+s}} \circ \dots \circ D^{(k_{1,1})}_{\rho|\cdot|^{B+s}}\right)(\pi),
\end{align*}
then
\[
\left(D^{(k_{2s,b})}_{\rho|\cdot|^{A-s+1}} \circ \dots \circ D^{(k_{2s,1})}_{\rho|\cdot|^{B-s+1}}\right)
\circ \dots \circ
\left(D^{(k_{1,b})}_{\rho|\cdot|^{A+s}} \circ \dots \circ D^{(k_{1,1})}_{\rho|\cdot|^{B+s}}\right)(\sigma) 
\not= 0.
\]
\end{enumerate}

\item
When $-(b-1)/2 \leq s < 0$, 
for an irreducible summand $\sigma$ of $u_\rho(a, b+2s) \rtimes \pi$, 
the following are equivalent. 
\begin{enumerate}
\item
$\soc(u_\rho(a, -2s)|\cdot|^{-\half{b}} \rtimes \sigma)$
is a subrepresentation of $u_\rho(a,b)|\cdot|^s \rtimes \pi$; 
\item
if 
\begin{align*}
&D^{\max}_{\rho|\cdot|^{B-s-1}, \dots, \rho|\cdot|^{-A-s-1}}
\circ \dots \circ
D^{\max}_{\rho|\cdot|^{B+s}, \dots, \rho|\cdot|^{-A+s}}(\pi)
\\&=
\left(D^{(k_{a,-2s})}_{\rho|\cdot|^{-A-s-1}} \circ \dots \circ D^{(k_{1,-2s})}_{\rho|\cdot|^{B-s-1}}\right)
\circ \dots \circ
\left(D^{(k_{a,1})}_{\rho|\cdot|^{-A+s}} \circ \dots \circ D^{(k_{1,1})}_{\rho|\cdot|^{B+s}}\right)(\pi),
\end{align*}
then 
\[
\left(D^{(k_{a,-2s})}_{\rho|\cdot|^{-A-s-1}} \circ \dots \circ D^{(k_{1,-2s})}_{\rho|\cdot|^{B-s-1}}\right)
\circ \dots \circ
\left(D^{(k_{a,1})}_{\rho|\cdot|^{-A+s}} \circ \dots \circ D^{(k_{1,1})}_{\rho|\cdot|^{B+s}}\right)(\sigma) 
\not= 0. 
\]
\end{enumerate}
\end{enumerate}
\end{prop}
\begin{proof}
We only prove (1). 
The proof of (2) is similar. 
From now, we assume that $0 < s \leq (a-1)/2$.
\par

Note that 
\begin{align*}
&D^{\max}_{\rho|\cdot|^{B-s+1}, \dots, \rho|\cdot|^{A-s+1}}
\circ \dots \circ
D^{\max}_{\rho|\cdot|^{B+s}, \dots, \rho|\cdot|^{A+s}}
\left( u_\rho(a,b)|\cdot|^s \rtimes \pi \right)
\\&= 
u_\rho(a-2s,b) \rtimes 
D^{\max}_{\rho|\cdot|^{B-s+1}, \dots, \rho|\cdot|^{A-s+1}}
\circ \dots \circ
D^{\max}_{\rho|\cdot|^{B+s}, \dots, \rho|\cdot|^{A+s}}(\pi)
\end{align*}
up to semisimplification.
Hence, if 
$\soc(u_\rho(2s, b)|\cdot|^{\half{a}} \rtimes \sigma)$
is a subrepresentation of $u_\rho(a,b)|\cdot|^s \rtimes \pi$, 
then we must have 
\[
\left(D^{(k_{2s,b})}_{\rho|\cdot|^{A-s+1}} \circ \dots \circ D^{(k_{2s,1})}_{\rho|\cdot|^{B-s+1}}\right)
\circ \dots \circ
\left(D^{(k_{1,b})}_{\rho|\cdot|^{A+s}} \circ \dots \circ D^{(k_{1,1})}_{\rho|\cdot|^{B+s}}\right)(\sigma) 
\not= 0.
\]
This shows that (a) implies (b).
\par

If we set $\pi' = \soc(u_\rho(2s, b)|\cdot|^{\half{a}} \rtimes \sigma)$, 
then 
\[
\pi' \hookrightarrow u_\rho(2s, b)|\cdot|^{\half{a}} \times u_\rho(a-2s,b) \rtimes \pi. 
\]
Note that $u_\rho(a,b)|\cdot|^s$ is a unique irreducible subrepresentation 
of $u_\rho(2s, b)|\cdot|^{\half{a}} \times u_\rho(a-2s,b)$, 
which is characterized among its composition series
by
\[
\left(L_{\rho|\cdot|^{A-s+1}} \circ \dots \circ L_{\rho|\cdot|^{B-s+1}}\right)
\circ \dots \circ
\left(L_{\rho|\cdot|^{A+s}} \circ \dots \circ L_{\rho|\cdot|^{B+s}}\right)
\left(u_\rho(a,b)|\cdot|^s\right) \not= 0,
\]
where $L_{\rho|\cdot|^x}$ is the left $\rho|\cdot|^x$-derivative, 
which is an analogue of $D_{\rho|\cdot|^x}$ for general linear groups
(cf., see \S \ref{der.gl} below).
Now if we assume (b), 
by considering the exponents of 
$D^{\max}_{\rho|\cdot|^{B-s+1}, \dots, \rho|\cdot|^{A-s+1}}
\circ \dots \circ
D^{\max}_{\rho|\cdot|^{B+s}, \dots, \rho|\cdot|^{A+s}}(\pi')$
(cf., see Example \ref{ex.35}), 
we see that the inclusion 
$\pi' \hookrightarrow u_\rho(2s, b)|\cdot|^{\half{a}} \times u_\rho(a-2s,b) \rtimes \pi$
factors through $\pi' \hookrightarrow u_\rho(a,b)|\cdot|^s \rtimes \pi$.
Hence we obtain (a). 
This completes the proof.
\end{proof}

%\section{Unitary inductions}
%\section{Unitary inductions}
\section{Unitary inductions}\label{s.unitary}
Let
\begin{itemize}
\item
$\rho \in \Cusp^\bot(\GL_d(F))$; 
\item
$\pi \in \Irr(G_n)$ be of Arthur type; and
\item
$a$ and $b$ be positive integers. 
\end{itemize}
In the previous section, 
we reduce the study of $\soc(u_\rho(a,b)|\cdot|^s \rtimes \pi)$ for $s \in \R$
to the case where $s = 0$. 
In this section, we treat this case. 
To do this, we recall terminologies of $A$-parameters and $A$-packets. 

%\subsection{$A$-parameters}
\subsection{$A$-parameters}\label{Apar}
Denote by $\widehat{G}_n$ the complex dual group of $G_n$. 
Namely, $\widehat{G}_n = \Sp_{2n}(\C)$ if $G_n = \SO_{2n+1}(F)$, and 
$\widehat{G}_n = \SO_{2n+1}(\C)$ if $G_n = \Sp_{2n}(F)$.  
Recall that an \emph{$A$-parameter for $G_n$} is 
the $\widehat{G}_n$-conjugacy class of an admissible homomorphism
\[
\psi \colon W_F \times \SL_2(\C) \times \SL_2(\C) \rightarrow \widehat{G}_n
\]
such that the image of the Weil group $W_F$ is bounded.
By composing with the standard representation of $\widehat{G}_n$, 
we can regard $\psi$ as a representation of $W_F \times \SL_2(\C) \times \SL_2(\C)$. 
It decomposes as 
\[
\psi = \bigoplus_\rho\left(\bigoplus_{i \in I_\rho} \rho \boxtimes S_{a_i} \boxtimes S_{b_i}\right), 
\]
where 
\begin{itemize}
\item
$\rho$ runs over $\cup_{d \geq 1} \Cusp_\unit(\GL_d(F))$,
which is identified with an irreducible bounded representation of $W_F$ 
by the local Langlands correspondence for the general linear groups; 
\item
$S_a$ is the unique irreducible algebraic representation of $\SL_2(\C)$ of dimension $a$.
\end{itemize}
Notice that $a_i$ and $b_i$ depend on $\rho$, but we do not write it.
We write $\rho \boxtimes S_a = \rho \boxtimes S_a \boxtimes S_1$ 
and $\rho = \rho \boxtimes S_1 \boxtimes S_1$ for short. 
\par

Let $\psi$ be as above. 
We say that 
$\psi$ is \emph{of good parity} 
if $\rho \boxtimes S_{a_i} \boxtimes S_{b_i}$ is self-dual of the same type as $\psi$ 
for any $\rho$ and $i \in I_\rho$, 
i.e., 
\begin{itemize}
\item
$\rho \in \Cusp^\bot(\GL_d(F))$ is orthogonal and $a_i+b_i \equiv 0 \bmod 2$ 
if $G_n = \Sp_{2n}(F)$
(\resp $a_i+b_i \equiv 1 \bmod 2$ if $G_n = \SO_{2n+1}(F)$); or 
\item
$\rho \in \Cusp^\bot(\GL_d(F))$ is symplectic and $a_i+b_i \equiv 1 \bmod 2$ 
if $G_n = \Sp_{2n}(F)$
(\resp $a_i+b_i \equiv 0 \bmod 2$ if $G_n = \SO_{2n+1}(F)$).
\end{itemize}
Let $\Psi(G_n) \supset \Psi_\gp(G_n)$ be the sets of equivalence classes of $A$-parameters
and $A$-parameters of good parity, respectively. 
Also, we set $\Phi_\temp(G_n)$ to be the subset of $\Psi(G_n)$ consisting of
\emph{tempered} $A$-parameters, i.e., 
$A$-parameters $\phi$ which are trivial on the second $\SL_2(\C)$. 
Finally, we set $\Phi_\gp(G_n) = \Psi_\gp(G_n) \cap \Phi_\temp(G_n)$.
\par

For 
$\psi = \oplus_\rho(\oplus_{i \in I_\rho} \rho \boxtimes S_{a_i} \boxtimes S_{b_i}) \in \Psi_\gp(G_n)$, 
define the \emph{enhanced component group} by
\[
\AA_\psi = \bigoplus_{\rho} \bigoplus_{i \in I_\rho} (\Z/2\Z) \alpha_{\rho, i}, 
\]
i.e., $\AA_\psi$ is a $(\Z/2\Z)$-vector space 
with a canonical basis $\alpha_{\rho, i}$ corresponding to 
$\rho \boxtimes S_{a_i} \boxtimes S_{b_i}$.  
The \emph{component group} $\Sc_\psi$ is 
the quotient of $\AA_\psi$ by the subgroup generated by 
\begin{itemize}
\item
$\alpha_{\rho, i} + \alpha_{\rho, j}$ 
such that $\rho \boxtimes S_{a_i} \boxtimes S_{b_i} = \rho \boxtimes S_{a_j} \boxtimes S_{b_j}$; 
and
\item
$z_\psi = \sum_{\rho}\sum_{i \in I_\rho} \alpha_{\rho,i}$, 
which is called the \emph{central element} of $\AA_\psi$.
\end{itemize}
Let $\widehat{\Sc_\psi} \subset \widehat{\AA_\psi}$ 
be the Pontryagin duals of $\Sc_\psi$ and $\AA_\psi$, respectively. 
When $\ep \in \widehat{\AA_\psi}$, 
we write $\ep(\rho \boxtimes S_{a_i} \boxtimes S_{b_i}) \coloneqq \ep(\alpha_{\rho, i}) \in \{\pm1\}$.

%\subsection{$A$-packets}\label{Apacket}
\subsection{$A$-packets}\label{Apacket}
Let $\Irr_\unit(G_n)$ (\resp $\Irr_\temp(G_n)$) be the set of
equivalence classes of irreducible unitary (\resp tempered) representations of $G_n$.
To an $A$-parameter $\psi \in \Psi(G_n)$, 
Arthur \cite[Theorem 1.5.1 (a)]{Ar} associated an $A$-packet $\Pi_\psi$, 
which is a finite multi-set over $\Irr_\unit(G_n)$.
We say that $\pi \in \Irr(G_n)$ is \emph{of Arthur type} if $\pi \in \Pi_\psi$ for some $\psi \in \Psi(G_n)$. 
In particular, such $\pi$ is unitary. 
\par

M{\oe}glin \cite{Moe11c} showed that $\Pi_\psi$ is multiplicity-free, 
i.e., a subset of $\Irr_\unit(G_n)$. 
By \cite[Theorem 1.5.1 (b)]{Ar}, if $\phi \in \Phi_\temp(G_n)$ is a tempered $A$-parameter, 
then $\Pi_\phi$ is a subset of $\Irr_\temp(G_n)$ and
\[
\Irr_\temp(G_n) = \bigsqcup_{\phi \in \Phi_\temp(G_n)} \Pi_\phi
\quad \text{(disjoint union)}.
\]
However, $\Pi_{\psi_1} \cap \Pi_{\psi_2} \not= \emptyset$ 
even if $\psi_1 \not\cong \psi_2$ in general.
\par

If $\psi = \oplus_\rho (\oplus_{i \in I_\rho} \rho \boxtimes S_{a_i} \boxtimes S_{b_i})$, 
set 
\[
\tau_{\psi} = 
\bigtimes_\rho\bigtimes_{i \in I_\rho} u_{\rho}(a_i,b_i)
=
\bigtimes_\rho\bigtimes_{i \in I_\rho} 
\begin{pmatrix}
\half{a_i-b_i} & \ldots & \half{a_i+b_i}-1 \\
\vdots & \ddots  & \vdots \\
-\half{a_i+b_i}+1 & \ldots & -\half{a_i-b_i} 
\end{pmatrix}_{\rho}
\]
to be a product of (unitary) Speh representations, 
which is an irreducible unitary representation of $\GL_m(F)$ with $m = \dim(\psi)$.
\par

\begin{prop}[{\cite[Theorem 6]{Moe06a}, \cite[Proposition 8.11]{X2}}]\label{bad}
Any $\psi \in \Psi(G_n)$ can be decomposed as 
\[
\psi = \psi_1 \oplus \psi_0 \oplus \psi_1^\vee, 
\]
where 
\begin{itemize}
\item
$\psi_0 \in \Psi_\gp(G_{n_0})$; 
\item
$\psi_1$ is a direct sum of irreducible representations of $W_F \times \SL_2(\C) \times \SL_2(\C)$
which are not self-dual of the same type as $\psi$. 
\end{itemize}
For $\pi_0 \in \Pi_{\psi_0}$, 
the parabolically induced representation $\tau_{\psi_1} \rtimes \pi_0$
is irreducible and independent of the choice of $\psi_1$. 
Moreover, 
\[
\Pi_\psi = \left\{ \tau_{\psi_1} \rtimes \pi_0 \;\middle|\; \pi_0 \in \Pi_{\psi_0} \right\}.
\]
\end{prop}
\par

Through this paper, we implicitly fix a Whittaker datum for $G_n$. 
Let $\psi \in \Psi_\gp(G_n)$ so that we have defined the component group $\Sc_\psi$. 
Arthur \cite[Theorem 1.5.1 (a)]{Ar} gave a map
\[
\Pi_\psi \rightarrow \widehat{\Sc_\psi},\; \pi \mapsto \pair{\cdot, \pi}_\psi. 
\]
If $\psi = \phi \in \Phi_\gp(G_n)$ is tempered, this map is bijective. 
When $\pi \in \Pi_\phi$ corresponds to $\ep \in \widehat{\Sc_\phi}$, we write $\pi = \pi(\phi, \ep)$.

\begin{prop}\label{lem223}
Let $\psi \in \Psi_{\gp}(G_n)$. 
Suppose that $\psi = \psi_0 \oplus (\rho \boxtimes S_a \boxtimes S_b)^{\oplus 2}$. 
Then for $\ep_0 \in \widehat{\Sc_{\psi_0}}$, we have
\[
\bigoplus_{\substack{\pi_0 \in \Pi_{\psi_0} \\ \pair{\cdot, \pi_0}_{\psi_0} = \ep_0}} 
u_{\rho}(a,b) \rtimes \pi_0 
= \bigoplus_{\substack{\pi \in \Pi_\psi \\ \pair{\cdot, \pi}_\psi|_{S_{\psi_0}} = \ep_0}}
\pi.
\]
In particular, $u_{\rho}(a,b) \rtimes \pi_0$ is multiplicity-free. 
\end{prop}
\begin{proof}
See (the proof of) \cite[Proposition 2.4.3]{Ar}. 
\end{proof}

%\subsection{Extended multi-segments}
\subsection{Extended multi-segments}\label{s.EMS}
To describe $A$-packets, in \cite{At}, we introduced the following notion.

\begin{defi}\label{segments}
\begin{enumerate}
\item
An \emph{extended segment} is a triple $([A,B]_\rho, l, \eta)$,
where
\begin{itemize}
\item
$[A,B]_\rho = \{\rho|\cdot|^A, \dots, \rho|\cdot|^B \}$ is a segment; 
\item
$l \in \Z$ with $0 \leq l \leq \half{b}$, where $b \coloneqq \#[A,B]_\rho = A-B+1$; 
\item
$\eta \in \{\pm1\}$. 
\end{itemize}

\item
Two extended segments $([A,B]_\rho, l, \eta)$ and $([A',B']_{\rho'}, l', \eta')$ are \emph{equivalent} 
if 
\begin{itemize}
\item
$[A,B]_\rho = [A',B']_{\rho'}$; 
\item
$l = l'$; and
\item
$\eta = \eta'$ whenever $l = l' < \half{b}$. 
\end{itemize}
Similarly, two multi-sets of extended segments 
$\{ ([A_i,B_i]_{\rho_i}, l_i, \eta_i) \}_{i \in I}$ 
and $\{ ([A'_i,B'_i]_{\rho}, l'_i, \eta'_i) \}_{i \in I}$ 
with the same index set $I$ are \emph{equivalent}
if $([A_i,B_i]_\rho, l_i, \eta_i)$ and $([A'_i,B'_i]_{\rho}, l'_i, \eta'_i)$ are equivalent for all $i \in I$.

\item
An \emph{extended multi-segment} for $G_n$ is an equivalence class of multi-sets of extended segments 
\[
\EE = \bigcup_{\rho}\{ ([A_i,B_i]_{\rho}, l_i, \eta_i) \}_{i \in (I_\rho,>)}
\]
such that 
\begin{itemize}
\item
$\rho$ runs over $\cup_{d \geq 1}\Cusp^\bot(\GL_d(F))$;
\item
$I_\rho$ is a totally ordered finite set with a fixed order $>$ which is called admissible;

\item
$A_i + B_i \geq 0$ for all $\rho$ and $i \in I_\rho$; 

\item
as a representation of $W_F \times \SL_2(\C) \times \SL_2(\C)$, 
\[
\psi_\EE \coloneqq 
\bigoplus_\rho \bigoplus_{i \in I_\rho} \rho \boxtimes S_{a_i} \boxtimes S_{b_i}
\]
belongs to $\Psi_\gp(G_n)$, 
where $a_i \coloneqq A_i+B_i+1$ and $b_i \coloneqq A_i-B_i+1$; 

\item
a sign condition
\[
\prod_{\rho} \prod_{i \in I_\rho} (-1)^{[\half{b_i}]+l_i} \eta_i^{b_i} = 1
\]
holds. 
\end{itemize}

\end{enumerate}
\end{defi}

In \cite{At}, to an extended multi-segment $\EE$ for $G_n$, 
we associate a representation $\pi(\EE)$ of $G_n$. 
To describe $\pi(\EE)$ explicitly, 
it was important to consider several orders on $I_\rho$. 
Nevertheless, in this paper, we assume that the order $>$ on $I_\rho$ always satisfies that 
\[
B_i < B_j \implies i < j.
\]
\par

We review the definition of $\pi(\EE)$. 
Let $\EE = \cup_{\rho}\{ ([A_i,B_i]_{\rho}, l_i, \eta_i) \}_{i \in (I_\rho,>)}$.
We say that 
\begin{itemize}
\item
$\EE$ has a \emph{discrete diagonal restriction (DDR)} 
if, for any $\rho$ and $i,j \in I_\rho$ with $i \not= j$, 
the segments $[A_i,B_i]_\rho$ and $[A_j,B_j]_\rho$ have no intersection; 

\item
$\EE$ is \emph{non-negative} if $B_i \geq 0$ for any $\rho$ and $i \in I_\rho$.  
\end{itemize}
When $\EE$ is non-negative DDR, we define 
\begin{align*}
\pi(\EE) = \soc \left(
\bigtimes_\rho \bigtimes_{i \in I_\rho}
\begin{pmatrix}
B_i & \ldots & B_i + l_i -1\\
\vdots & \ddots & \vdots \\
-A_i  & \ldots & -(A_i-l_i+1)
\end{pmatrix}_\rho
\rtimes
\pi(\phi, \ep)
\right)
\end{align*}
with 
\[
\phi = \bigoplus_\rho \bigoplus_{i \in I_\rho} 
\rho \boxtimes \left( S_{2(B_i+l_i)+1} \oplus \dots \oplus S_{2(A_i-l_i)+1} \right)
\]
and $\ep(\rho \boxtimes S_{2(B_i+l_i+k)+1}) = (-1)^k \eta_i$ for $0 \leq k \leq b_i - 2l_i - 1$.
In general, 
take a sequence of non-negative integers $\cup_\rho\{t_i\}_{i \in (I_\rho,>)}$ such that 
$\EE' = \cup_\rho \{([A_i+t_i,B_i+t_i]_{\rho}, l_i, \eta_i)\}_{i \in (I_\rho,>)}$ is non-negative DDR, 
and define 
\[
\pi(\EE) = 
\circ_\rho \circ_{i \in I_\rho}
\left(
D_{\rho|\cdot|^{B_i+1}, \dots, \rho|\cdot|^{A_i+1}}
\circ \dots \circ 
D_{\rho|\cdot|^{B_i+t_i}, \dots, \rho|\cdot|^{A_i+t_i}}
\right)
(\pi(\EE')).
\]
This definition does not depend on the choice of $\cup_\rho\{t_i\}_{i \in (I_\rho,>)}$.
\par

Note that $\pi(\EE)$ is irreducible or zero by Theorems \ref{der} and \ref{der2}. 
The following properties were proven in \cite[Theorems 1.2, 1.3, 1.4, 3.5]{At}: 
\begin{itemize}
\item
There exists a non-vanishing criterion for $\pi(\EE)$. 

\item
For $\psi = \oplus_{\rho}(\oplus_{i \in I_\rho} \rho \boxtimes S_{a_i} \boxtimes S_{b_i}) 
\in \Psi_\gp(G_n)$
with a fixed order $>$ on $I_\rho$, we have 
\[
\Pi_\psi = \{\pi(\EE) \;|\; \psi_\EE \cong \psi\} \setminus \{0\}. 
\]

\item
The character $\pair{\cdot, \pi(\EE)}_{\psi_\EE}$ is explicitly determined by $\EE$. 
\end{itemize}
\par

%\subsection{Decompositions of unitary inductions}
\subsection{Decompositions of unitary inductions}
Now we describe the unitary induction $u_\rho(a,b) \rtimes \pi$ 
for $\pi$ of Arthur type, i.e., $\pi \in \Pi_\psi$ for some $\psi \in \Psi(G_n)$. 
We decompose 
$\psi = \psi_1 \oplus \psi_0 \oplus \psi_1^\vee$
as in Proposition \ref{bad}. 
According to this proposition, $\pi = \tau_{\psi_1} \rtimes \pi_0$ for some $\pi_0 \in \Pi_{\psi_0}$. 
Since $u_\rho(a,b)$ and $\tau_{\psi_1}$ are both unitary,  
we have $u_\rho(a,b) \times \tau_{\psi_1} \cong \tau_{\psi_1} \times u_\rho(a,b)$. 
Hence the problem is reduced to the case where $\psi = \psi_0 \in \Psi_\gp(G_n)$.
In this case, one can write $\pi = \pi(\EE)$ for some extended multi-segment $\EE$ for $G_n$. 
\par

When $\psi_\EE \oplus (\rho \boxtimes S_a \boxtimes S_b)^{\oplus2}$ is not of good parity, 
by Proposition \ref{bad}, $u_\rho(a,b) \rtimes \pi(\EE)$ is irreducible. 
Otherwise, we have the following: 

\begin{thm}\label{s=0}
Suppose that $\psi_\EE \oplus (\rho \boxtimes S_a \boxtimes S_b)^{\oplus2}$ is of good parity.
For $(l,\eta) \in \Z \times \{\pm1\}$ with $0 \leq l \leq b/2$, 
define $\EE_{(l,\eta)}$ by adding $([A,B]_{\rho},l,\eta)$ and $([A,B]_\rho,l,(-1)^{A-B}\eta)$ to $\EE$, 
where $A = \half{a+b}-1$ and $B = \half{a-b}$,
such that if we let $i_0$ (\resp $i'_0$) be the index 
for $([A,B]_{\rho},l,\eta)$ (\resp $([A,B]_\rho,l,(-1)^{A-B}\eta)$), 
then $i_0 <i'_0$ are adjacent, and $j > i'_0$ if and only if $B_j > B$. 
Then
\[
u_\rho(a,b) \rtimes \pi(\EE) \cong \bigoplus_{(l,\eta)}\pi(\EE_{(l,\eta)}), 
\]
where $(l,\eta)$ runs over the set $\{(l,\eta) \in \Z \times \{\pm1\} \;|\; 0 \leq l \leq b/2\}/\sim$. 
Here, we write $(l,\eta) \sim (l',\eta')$ if $l=l'$ and if $\eta=\eta'$ whenever $l=l' < b/2$. 
\end{thm}
\begin{proof}
This theorem seems to be already known by M{\oe}glin
(watch the video of her talk \cite{M-video}). 
For the sake of completeness, 
we give a proof. 
\par

Fix $\psi \in \Psi_\gp(G_n)$ and set $\psi' = \psi \oplus (\rho \boxtimes S_a \boxtimes S_b)^{\oplus2}$. 
Since 
\[
\bigoplus_{\substack{\EE \\ \psi_\EE \cong \psi}}u_\rho(a,b) \rtimes \pi(\EE)
\cong \bigoplus_{\pi' \in \Pi_{\psi'}}\pi' 
\cong \bigoplus_{\substack{\EE \\ \psi_\EE \cong \psi}}\bigoplus_{(l,\eta)}\pi(\EE_{(l,\eta)}), 
\]
it is enough to show that for any fixed $\psi$, 
one of two inclusions 
$u_\rho(a,b) \rtimes \pi(\EE) \subset \oplus_{(l,\eta)}\pi(\EE_{(l,\eta)})$
or 
$\oplus_{(l,\eta)}\pi(\EE_{(l,\eta)}) \subset u_\rho(a,b) \rtimes \pi(\EE)$
holds whenever $\psi_\EE \cong \psi$. 
We will prove this by considering several steps. 
Write $\EE = \cup_{\rho'}\{ ([A_i,B_i]_{\rho'}, l_i, \eta_i) \}_{i \in (I_{\rho'},>)}$.
We may assume that $\pi(\EE_{(l,\eta)}) \not= 0$.

\begin{enumerate}
\item
We assume that $\psi = \psi_\EE$ is a tempered $A$-parameter, i.e., 
$A_i = B_i$ for all $\rho'$ and $i \in I_{\rho'}$. 
In this case, since the map 
$\Pi_\psi \rightarrow \widehat{\Sc_\psi}$, $\pi \mapsto \pair{\cdot, \pi}_\psi$ 
is injective, 
to prove $\pi(\EE_{(l,\eta)}) \subset u_\rho(a,b) \rtimes \pi(\EE)$,
by Proposition \ref{lem223}, it suffices to check that 
$\pair{\cdot, \pi(\EE_{(l,\eta)})}_{\psi'} |_{\Sc_\psi} = \pair{\cdot, \pi(\EE)}_\psi$. 
It follows from \cite[Theorem 3.5]{At}. 

\item
Define $\EE'_{(l,\eta)}$ by adding 
$([A,B]_{\rho},l,\eta)$ and $([A+b,B+b]_\rho,l,(-1)^{A-B}\eta)$ to $\EE$ 
as adjacent elements similar to $\EE_{(l,\eta)}$. 
We assume that $\EE'_{(l,\eta)}$ is non-negative DDR. 
In this case, 
\[
\pi(\EE_{(l,\eta)}) = \soc\left(
\tau
\rtimes
\pi(\FF_{(l,\eta)})
\right),
\]
where $\FF_{(l,\eta)}$ is defined from $\EE_{(l,\eta)}$ 
by replacing $([A_i,B_i]_{\rho'},l_i,\eta_i)$ with $([A_i-l_i,B_i+l_i]_{\rho'},0,\eta_i)$
for all $\rho'$ and $i \in I_{\rho'}$, 
and we set 
\[
\tau = \bigtimes_{\rho'} \bigtimes_{i \in I_{\rho'}}
\begin{pmatrix}
B_i & \ldots & B_i + l_i -1\\
\vdots & \ddots & \vdots \\
-A_i  & \ldots & -(A_i-l_i+1)
\end{pmatrix}_{\rho'}.
\]
Moreover, by replacing 
$\{([A_i-l_i,B_i+l_i]_{\rho'},0,\eta_i)\}$ with 
\[
\bigcup_{k=0}^{A_i-B_i-2l_i} \{ ([B_i+l_i+k,B_i+l_i+k]_{\rho'}, 0, (-1)^{k}\eta_i) \}
\]
and vice versa, 
one can apply the first case to $\pi(\FF_{(l,\eta)})$. 
Hence $\oplus_{(l,\eta)}\pi(\FF_{(l,\eta)}) \cong u_\rho(a,b) \rtimes \pi(\FF)$
with $\FF = \cup_{\rho'}\{ ([A_i-l_i,B_i+l_i]_{\rho'}, 0, \eta_i) \}_{i \in (I_{\rho'},>)}$.
Since $[B_i+l_i-1, -A_i]_{\rho'}$ and $[A,-A]_\rho$ are always not linked, 
by \cite[Theorem 1.1]{T-irr}, we have $\tau \times u_\rho(a,b) \cong u_\rho(a,b) \times \tau$. 
Hence we have
\begin{align*}
\bigoplus_{(l,\eta)}\pi(\EE_{(l,\eta)}) 
&\cong
\soc\left(
\tau \rtimes \bigoplus_{(l,\eta)}\pi(\FF_{(l,\eta)})
\right)
\\&\cong 
\soc\left(
\tau \times u_\rho(a,b) \rtimes \pi(\FF)
\right)
\\&\cong 
\soc\left(
u_\rho(a,b) \times \tau \rtimes \pi(\FF)
\right).
\end{align*}
Since $\pi(\EE) = \soc(\tau \rtimes \pi(\FF))$, we have 
\[
u_\rho(a,b) \rtimes \pi(\EE) \hookrightarrow 
u_\rho(a,b) \times \tau \rtimes \pi(\FF).
\] 
Since the left hand side is semisimple, 
we see that $u_\rho(a,b) \rtimes \pi(\EE) \hookrightarrow \oplus_{(l,\eta)}\pi(\EE_{(l,\eta)})$, 
as desired. 

\item
We consider the general case. 
Take $\EE_{(l,\eta)}' = \cup_{\rho'}\{ ([A_i+t_i,B_i+t_i]_{\rho'}, l_i, \eta_i) \}_{i \in (I'_{\rho'},>)}$
such that it is non-negative DDR and 
\[
\pi(\EE_{(l,\eta)}) = 
\circ_{\rho'} \circ_{i \in I'_{\rho'}}
\left(
D_{\rho'|\cdot|^{B_i+1}, \dots, \rho'|\cdot|^{A_i+1}}
\circ \dots \circ 
D_{\rho'|\cdot|^{B_i+t_i}, \dots, \rho'|\cdot|^{A_i+t_i}}
\right)
(\pi(\EE_{(l,\eta)}')).
\]
By construction, we have $I'_{\rho'} = I_{\rho'}$ unless $\rho' \cong \rho$, 
and in this case, $I'_\rho = I_\rho \cup \{i_0,i'_0\}$ such that 
$i_0 < i_0'$ are adjacent and that 
\begin{align*}
([A_{i_0}+t_{i_0},B_{i_0}+t_{i_0}]_{\rho}, l_{i_0}, \eta_{i_0}) &= ([A+t,B+t]_\rho,l,\eta), \\
([A_{i'_0}+t_{i'_0},B_{i'_0}+t_{i'_0}]_{\rho}, l_{i'_0}, \eta_{i'_0}) &= ([A+t',B+t']_\rho,l,(-1)^{A-B}\eta)
\end{align*}
for some $t < t'$.
By the same argument as \cite[Corollary 5.3]{At}, 
we may reset $t' = t$ by replacing 
\begin{align*}
\left(
D_{\rho|\cdot|^{B+1}, \dots, \rho|\cdot|^{A+1}}
\circ \dots \circ 
D_{\rho|\cdot|^{B+t'}, \dots, \rho|\cdot|^{A+t'}}
\right)
\circ
\left(
D_{\rho|\cdot|^{B+1}, \dots, \rho|\cdot|^{A+1}}
\circ \dots \circ 
D_{\rho|\cdot|^{B+t}, \dots, \rho|\cdot|^{A+t}}
\right)
\end{align*}
with 
\begin{align*}
\left(
D^{\max}_{\rho|\cdot|^{B+1}, \dots, \rho|\cdot|^{A+1}}
\circ \dots \circ 
D^{\max}_{\rho|\cdot|^{B+t}, \dots, \rho|\cdot|^{A+t}}
\right).
\end{align*}
Then we can use the second case so that 
$\pi(\EE_{(l,\eta)}') \hookrightarrow u_\rho(a+2t,b) \rtimes \pi(\EE')$, 
where $\EE' = \cup_{\rho'}\{ ([A_i+t_i,B_i+t_i]_{\rho'}, l_i, \eta_i) \}_{i \in (I_{\rho'},>)}$.
By computing derivatives, 
we conclude that $\pi(\EE_{(l,\eta)}) \hookrightarrow u_\rho(a,b) \rtimes \pi(\EE)$. 
\end{enumerate}
This completes the proof. 
\end{proof}

Combining Propositions \ref{s>>0}, \ref{middle}, \ref{image} and Theorem \ref{s=0}, 
we obtain Theorem \ref{main}.

\begin{cor}\label{len}
The length of $u_\rho(a,b) \rtimes \pi(\EE)$ is at most $\min\{a,b\}+1$.
\end{cor}
\begin{proof}
By \cite[Theorem 1.3]{At}, if $\pi(\EE_{(l,\eta)}) \not= 0$
then $B+l \geq 0$, or $B+l = -1/2$ and $\eta \in \{\pm1\}$ is uniquely determined.
By a case-by-case calculation, 
we see that there are at most $\min\{a,b\}+1$ such pairs $(l,\eta)$. 
\end{proof}

\begin{cor}
Suppose that $\pi(\EE) \not= 0$. 
If $\EE$ contains $([A,B]_\rho,l_0,\eta_0)$ for some $(l_0,\eta_0)$, 
where $A = (a+b)/2-1$ and $B = (a-b)/2$, 
then $u_\rho(a,b) \rtimes \pi(\EE)$ is irreducible.
\end{cor}
\begin{proof}
If $\pi(\EE_{(l,\eta)}) \not= 0$, 
by \cite[Proposition 4.1]{At}, 
we see that $l=l_0$, and that $\eta$ is determined uniquely. 
\end{proof}

As in the following example, Corollary \ref{len} is optimum. 
\begin{ex}
Suppose that $\psi = (\rho \boxtimes S_a \boxtimes S_b)^{\oplus 2} \in \Psi_\gp(\SO_{2abd+1}(F))$. 
Then 
\begin{align*}
\bigoplus_{\pi \in \Pi_\psi}\pi 
&\cong u_\rho(a,b) \rtimes \1_{\SO_1(F)}
\\&\cong \bigoplus_{(l,\eta)}\pi( \{ ([A,B]_\rho, l, \eta), ([A,B]_\rho,l, (-1)^{A-B}\eta) \} ). 
\end{align*}
By \cite[Theorems 1.3, 1.4]{At}, $\pi( \{ ([A,B]_\rho, l, \eta), ([A,B]_\rho,l, (-1)^{A-B}\eta) \} ) \not= 0$
if and only if $B+l \geq 0$, or $B+l = -1/2$ and $\eta = +1$.
By a case-by-case calculation, 
we conclude that the length of $u_\rho(a,b) \rtimes \1_{\SO_1(F)}$ is equal to $\min\{a,b\}+1$.
\end{ex}

%\section{Irreducibility and examples}
%\section{Irreducibility and examples}
\section{Irreducibility and examples}\label{s.irred}
In this section, we discuss when $u_\rho(a,b)|\cdot|^s \rtimes \pi$ is irreducible. 
Also, we give some examples.

%\subsection{Irreducibility}
\subsection{Irreducibility}
We give some consequences of the results in the previous sections.

\begin{cor}\label{free}
Let $\pi \in \Irr(G_n)$ be of Arthur type. 
Then for any $s \in \R$, 
any irreducible subrepresentation of $u_\rho(a,b)|\cdot|^s \rtimes \pi$
appears in the semisimplification $[u_\rho(a,b)|\cdot|^s \rtimes \pi]$ with multiplicity one. 
In particular, the socle $\soc(u_\rho(a,b)|\cdot|^s \rtimes \pi)$ is multiplicity-free. 
\end{cor}
\begin{proof}
It follows from Propositions \ref{s>>0}, \ref{middle} and \ref{lem223}. 
\end{proof}

This corollary gives a criterion for the irreducibility. 
\begin{cor}\label{irred}
Let $\pi \in \Irr(G_n)$ be of Arthur type. 
Then $u_\rho(a,b)|\cdot|^s \rtimes \pi$ is irreducible if and only if all of the following conditions hold: 
\begin{itemize}
\item
$\soc(u_\rho(a,b)|\cdot|^s \rtimes \pi)$ is irreducible; 
\item
$\soc(u_\rho(a,b)|\cdot|^{-s} \rtimes \pi)$ is irreducible; 
\item
$\soc(u_\rho(a,b)|\cdot|^s \rtimes \pi) \cong \soc(u_\rho(a,b)|\cdot|^{-s} \rtimes \pi)$.
\end{itemize}
\end{cor}
\begin{proof}
The only if part is trivial. 
To prove the if part, we assume the three conditions.
If $u_\rho(a,b)|\cdot|^s \rtimes \pi$ were to be reducible, 
since $\soc(u_\rho(a,b)|\cdot|^{-s} \rtimes \pi)$
is a unique irreducible quotient of $u_\rho(a,b)|\cdot|^s \rtimes \pi$, 
we would have 
\[
\frac{u_\rho(a,b)|\cdot|^s \rtimes \pi}
{\soc(u_\rho(a,b)|\cdot|^s \rtimes \pi)}
\twoheadrightarrow \soc(u_\rho(a,b)|\cdot|^{-s} \rtimes \pi).
\]
This contradicts that 
$\soc(u_\rho(a,b)|\cdot|^s \rtimes \pi) \cong \soc(u_\rho(a,b)|\cdot|^{-s} \rtimes \pi)$
appears in $[u_\rho(a,b)|\cdot|^s \rtimes \pi]$ with multiplicity one (Corollary \ref{free}).
\end{proof}

The following sufficient condition for the irreducibility is useful.
\begin{thm}\label{irred-bad}
Let $\psi \in \Psi_\gp(G_n)$ and $\pi \in \Pi_\psi$. 
Suppose one of the following:
\begin{itemize}
\item
$s \not\in (1/2)\Z$; 
\item
$s \in (1/2)\Z \setminus \Z$ and 
$\psi \oplus (\rho \boxtimes S_a \boxtimes S_b)^{\oplus2}$ is of good parity; 
\item
$s \in \Z$ and $\psi \oplus (\rho \boxtimes S_a \boxtimes S_b)^{\oplus2}$ is not of good parity. 
\end{itemize}
Then $u_\rho(a,b)|\cdot|^s \rtimes \pi$ is irreducible. 
\end{thm}
\begin{proof}
The case where $s = 0$ is Proposition \ref{bad}. 
Since the irreducibility of $u_\rho(a,b)|\cdot|^s \rtimes \pi$ is equivalent to 
the one of $u_\rho(a,b)|\cdot|^{-s} \rtimes \pi$, 
we may assume that $s > 0$. 
Note that by Propositions \ref{s>>0} and \ref{middle}, 
$\soc(u_\rho(a,b)|\cdot|^s \rtimes \pi)$ and $\soc(u_\rho(a,b)|\cdot|^{-s} \rtimes \pi)$
are both irreducible. 
Hence by Corollary \ref{irred}, 
it is enough to show that 
$\soc(u_\rho(a,b)|\cdot|^s \rtimes \pi) \cong \soc(u_\rho(a,b)|\cdot|^{-s} \rtimes \pi)$. 
We prove this claim by several steps. 

\begin{enumerate}
\item
We assume that $s \in (1/2)\Z$ and $b = 1$. 
Write $\pi = L(\Delta_{\rho_1}[x_1,y_1], \dots, \Delta_{\rho_r}[x_r,y_r]; \pi(\phi, \ep))$.
Since $\psi \in \Psi_\gp(G_n)$, we have $\phi \in \Phi_\gp(G_{n_0})$. 
By \cite[Theorem 9.7]{Z} and \cite[Th{\'e}or{\`e}me (i)]{MW} together with our assumption, 
we see that
\begin{itemize}
\item
$u_\rho(a,1)|\cdot|^{s} \times \Delta_{\rho_i}[x_i,y_i] 
\cong \Delta_{\rho_i}[x_i,y_i] \times u_\rho(a,1)|\cdot|^{s}$; 
\item
$u_\rho(a,1)|\cdot|^{-s} \times \Delta_{\rho_i}[x_i,y_i] 
\cong \Delta_{\rho_i}[x_i,y_i] \times u_\rho(a,1)|\cdot|^{-s}$; 
\item
$u_\rho(a,1)|\cdot|^{s} \rtimes \pi(\phi,\ep) \cong u_\rho(a,1)|\cdot|^{-s} \rtimes \pi(\phi,\ep)$. 
\end{itemize}
These isomorphisms and the characterization of 
$\soc(u_\rho(a,b)|\cdot|^{\pm s} \rtimes \pi)$ 
obtained in the proofs of Propositions \ref{s>>0} and \ref{middle}, 
we see that 
\[
\soc(u_\rho(a,b)|\cdot|^s \rtimes \pi) \cong \soc(u_\rho(a,b)|\cdot|^{-s} \rtimes \pi),
\] 
as desired.

\item
We assume that $s \in (1/2)\Z$ and $b \geq 2$.
We prove the claim by induction on $b$. 
Set $\pi' = \soc(u_\rho(a,b)|\cdot|^{s} \rtimes \pi)$. 
Note that 
$u_\rho(a,b)|\cdot|^s \hookrightarrow \Delta_\rho[B+s, -A+s] \times u_\rho(a,b-1)|\cdot|^{s+1/2}$
with $A = (a+b)/2-1$ and $B = (a-b)/2$.
By the induction hypothesis, we have
$\pi' \hookrightarrow \Delta_\rho[B+s, -A+s] \times u_\rho(a,b-1)|\cdot|^{-s-\half{1}} \rtimes \pi$. 
We set 
\begin{align*}
\tau &= \Delta_\rho[B+s, -A+s] \times u_\rho(a,b-1)|\cdot|^{-s-\half{1}}
\\&= \begin{pmatrix}
B+s \\ \vdots \\ -A+s
\end{pmatrix}_\rho
\times 
\begin{pmatrix}
B-s & \ldots & A-s-1 \\
\vdots & \ddots & \vdots \\
-A-s & \ldots & -B-s-1
\end{pmatrix}_\rho.
\end{align*}
If $\tau$ is irreducible, then 
$\pi' \hookrightarrow u_\rho(a,b-1)|\cdot|^{-s-1/2} \times \Delta_\rho[B+s,-A+s] \rtimes \pi$. 
In this case, by the previous case, we have 
\[
\pi'
\hookrightarrow 
u_\rho(a,b-1)|\cdot|^{-s-1/2} \times \Delta_\rho[A-s,-B-s] \rtimes \pi.
\]
By taking derivatives, we can conclude that 
$\pi' \cong \soc(u_\rho(a,b)|\cdot|^{-s} \rtimes \pi)$.
\par

Noting that $s > 0$, by \cite[Theorem 1.1]{T-irr}, 
$\tau$ is reducible if and only if 
$B+s > A-s-1$, $-A+s > -B-s-1$ and $-A+s \leq A-s$. 
In this case, $\tau$ is of length $2$, and 
the socle $\soc(\tau)$ is isomorphic to
\[
L(\Delta_\rho[B-s,-A-s], \dots, \Delta_\rho[A-s-2,-B-s-2], 
\Delta_\rho[A-s-1,-A-s], \Delta_\rho[B+s, -B-s-1]). 
\]
This fact follows from \cite[Lemma 2.7]{LT} by taking the Zelevinsky dual.
Now suppose that $\pi' \hookrightarrow \soc(\tau) \rtimes \pi$.
Then by the previous case, 
\begin{align*}
\pi' & \hookrightarrow 
u_\rho(a,b-2)|\cdot|^{-s-1} \times \Delta_\rho[A-s-1,-A-s] \times \Delta_\rho[B+s, -B-s-1] \rtimes \pi
\\&\cong
u_\rho(a,b-2)|\cdot|^{-s-1} \times \Delta_\rho[A-s-1,-A-s] \times \Delta_\rho[B+s+1, -B-s] \rtimes \pi.
\end{align*}
Since $B+s+1 > A-s$, by \cite[Theorem 1.1]{T-irr}, 
we see that $\rho|\cdot|^{B+s+1}$ commutes with 
$\Delta_\rho[A-s-1,-A-s]$ and $u_\rho(a,b-2)|\cdot|^{-s-1}$. 
This implies that $D_{\rho|\cdot|^{B+s+1}}(\pi') \not= 0$. 
This contradicts that $D_{\rho|\cdot|^{B+s+1}}(u_\rho(a,b)|\cdot|^s \rtimes \pi) = 0$. 
Therefore, we again have 
$\pi' \hookrightarrow u_\rho(a,b-1)|\cdot|^{-s-1/2} \times \Delta_\rho[B+s,-A+s] \rtimes \pi$,
which implies the claim.

\item
We assume that $s \not\in (1/2)\Z$.
Using 
$u_\rho(a,b) \hookrightarrow 
u_\rho(a,1)|\cdot|^{-\half{b-1}} \times \dots \times u_\rho(a,1)|\cdot|^{\half{b-1}}$, 
a similar argument to the first case works. 
In fact, we do not need to assume that $\psi$ is of good parity in this case.
\end{enumerate}
This completes the proof.
\end{proof}

Let $\pi \in \Irr(G_n)$ be of Arthur type. 
We denote the minimal non-negative real number $s$ 
such that $u_\rho(a,b)|\cdot|^s \rtimes \pi$ is reducible by $s_0$. 
We call $s_0$ the \emph{first reducibility point} for $u_\rho(a,b)|\cdot|^s \rtimes \pi$. 
As in \cite[Section 3 (b)]{T-ext}, for $0 \leq s < s_0$, 
the irreducible induction $u_\rho(a,b)|\cdot|^s \rtimes \pi$ is unitary. 
Moreover, by \cite[Section 3 (c)]{T-ext}, 
all irreducible constituents of $u_\rho(a,b)|\cdot|^{s_0} \rtimes \pi$ are also unitary. 
Therefore, to attack the unitary dual problem for classical groups, 
it is important to compute $s_0$. 

\begin{cor}\label{FRP}
Let $\pi \in \Irr(G_n)$ be of Arthur type. 
Then we can compute the first reducibility point $s_0$ for $u_\rho(a,b)|\cdot|^s \rtimes \pi$ algorithmically. 
\end{cor}
\begin{proof}
By Theorem \ref{irred-bad}, $s_0$ belongs to $(1/2)\Z$. 
Moreover, 
by computing $\soc(u_\rho(a,b)|\cdot|^s \rtimes \pi)$ and $\soc(u_\rho(a,b)|\cdot|^{-s} \rtimes \pi)$
using Propositions \ref{s>>0}, \ref{middle}, \ref{image} and Theorem \ref{s=0}, 
we can determine $s_0$ by Corollary \ref{irred}.
\end{proof}

%\subsection{Examples}
\subsection{Examples}
Now we give some examples. 
In this subsection, we set $\rho = \1_{\GL_1(F)}$.
When $\phi = \rho \boxtimes (S_{2x_1+1} \oplus \dots \oplus S_{2x_r+1})$ 
and $\ep(\rho \boxtimes S_{2x_i+1}) = \ep_i$, 
we write $\pi(\phi, \ep) = \pi(x_1^{\ep_1}, \dots, x_r^{\ep_r})$.

\begin{ex}
Let us consider 
\[
u_\rho(2,3)|\cdot|^{s} \rtimes \1_{\Sp_{0}(F)},
\]
which is a representation of $\Sp_{12}(F)$.
We compute the socle of this representation for $s = \pm1/2$. 

\begin{enumerate}
\item
When $s = 1/2$, by Proposition \ref{middle}, 
\[
\soc(u_\rho(2,3)|\cdot|^{\half{1}} \rtimes \1_{\Sp_{0}(F)})
\hookrightarrow 
Z_\rho[0,2] \times Z_\rho[-1,1] \rtimes \1_{\Sp_0(F)}. 
\]
Noting that $\1_{\Sp_0(F)} = \pi(\{([0,0]_\rho,0,1)\})$, 
by Theorem \ref{s=0}, we have 
\begin{align*}
Z_\rho[-1,1] \rtimes \1_{\Sp_0(F)} 
\cong &
\pi(\{([1,-1]_\rho,1,1), ([1,-1]_\rho,1,1), ([0,0]_\rho,0,1)\}) 
\\&\oplus \pi(\{([1,-1]_\rho,1,-1), ([1,-1]_\rho,1,-1), ([0,0]_\rho,0,1)\})
\\\cong &
L((\rho|\cdot|^{-1})^{2}; \pi(0^+,0^+,0^+)) 
\oplus L(\rho|\cdot|^{-1}; \pi(0^-,0^-,1^+)). 
\end{align*}
By considering Proposition \ref{image}, we have 
\begin{align*}
&\soc(u_\rho(2,3)|\cdot|^{\half{1}} \rtimes \1_{\Sp_{0}(F)})
\\&\cong 
\soc(Z_\rho[0,2] \rtimes L((\rho|\cdot|^{-1})^{2}; \pi(0^+,0^+,0^+)) )
\oplus
\soc(Z_\rho[0,2] \rtimes L(\rho|\cdot|^{-1}; \pi(0^-,0^-,1^+)) ) 
\\&\cong 
L(\rho|\cdot|^{-1}, \Delta_\rho[0,-2]; \pi(0^+,0^+,1^+)) 
\oplus
L(\Delta_\rho[0,-1]; \pi(0^-,1^-,2^+)).
\end{align*}

\item
When $s = -1/2$, by Proposition \ref{middle}, 
\[
\soc(u_\rho(2,3)|\cdot|^{-\half{1}} \rtimes \1_{\Sp_{0}(F)})
\hookrightarrow 
\Delta_\rho[-1,-2] \times u_\rho(2,2) \rtimes \1_{\Sp_0(F)}. 
\]
By Theorem \ref{s=0}, we have 
\begin{align*}
u_\rho(2,2) \rtimes \1_{\Sp_0(F)} 
\cong &
\pi(\{([0,0]_\rho,0,1),([1,0]_\rho,1,1),([1,0]_\rho,1,-1)\}) 
\\&\oplus \pi(\{([0,0]_\rho,0,1),([1,0]_\rho,0,1),([1,0]_\rho,0,-1)\}) 
\\\cong &
L(\Delta_\rho[0,-1]^{2}; \pi(0^+)) 
\oplus L(\Delta_\rho[0,-1]; \pi(0^+,0^+,1^+)). 
\end{align*}
By considering Proposition \ref{image}, we have 
\begin{align*}
&\soc(u_\rho(2,3)|\cdot|^{-\half{1}} \rtimes \1_{\Sp_{0}(F)})
\\&\cong 
\soc(\Delta_\rho[-1,-2] \rtimes L(\Delta_\rho[0,-1]^{2}; \pi(0^+)) )
\oplus
\soc(\Delta_\rho[-1,-2] \rtimes L(\Delta_\rho[0,-1]; \pi(0^+,0^+,1^+)) ) 
\\&\cong 
L(\Delta_\rho[-1,-2],\Delta_\rho[0,-1]^{2}; \pi(0^+))
\oplus
L(\Delta_\rho[-1,-2],\Delta_\rho[0,-1]; \pi(0^+,0^+,1^+)). 
\end{align*}
\end{enumerate}
In particular, we see that 
the length of $u_\rho(2,3)|\cdot|^{1/2} \rtimes \1_{\Sp_{0}(F)}$ is at least $4$. 
\end{ex}

\begin{ex}
Let us consider $a = b = 4$ and 
\[
\EE = \{([3,-1]_\rho,2,-1), ([3,1]_\rho,0,-1), ([2,2]_\rho,0,-1)\}.
\]
Note that 
\[
\pi(\EE) = L(\Delta_\rho[-1,-3], \Delta_\rho[0,-2], \Delta_\rho[2,-3]; \pi(1^-,1^-,2^+)). 
\]
We determine the first reducibility point $s_0$ for $u_\rho(4,4)|\cdot|^s \rtimes \pi(\EE)$.
To do this, we compute its socles for some $s \in \Z$.

\begin{enumerate}
\item
When $s = 0$, we have 
\[
u_\rho(4,4) \rtimes \pi(\EE) = 
\pi(\EE_{(2,+1)}^{(0)}) \oplus \pi(\EE_{(1,+1)}^{(0)}) 
\oplus \pi(\EE_{(1,-1)}^{(0)}) \oplus \pi(\EE_{(0,+1)}^{(0)}) \oplus \pi(\EE_{(0,-1)}^{(0)}), 
\]
where $\EE_{(l,\eta)}^{(0)} = \EE \cup \{([3,0]_\rho,l,\eta), ([3,0]_\rho,l,-\eta)\}$. 
By \cite[Theorems 1.3, 1.4]{At}, 
we have 
$\pi(\EE_{(2,+1)}^{(0)}) = \pi(\EE_{(1,+1)}^{(0)}) 
= \pi(\EE_{(0,+1)}^{(0)}) = \pi(\EE_{(0,-1)}^{(0)}) = 0$. 
Hence $u_\rho(4,4) \rtimes \pi(\EE) = \pi(\EE_{(1,-1)}^{(0)})$ is irreducible. 

\item
When $s = 1$, we have 
\[
u_\rho(4,4)|\cdot|^1 \rtimes \pi(\EE) 
\hookrightarrow u_\rho(2,4)|\cdot|^2 \times u_\rho(2,4) \rtimes \pi(\EE). 
\]
As in the previous case, 
we have $u_\rho(2,4) \rtimes \pi(\EE) = \pi(\EE_{(2,-1)}^{(1)}) \oplus \pi(\EE_{(1,-1)}^{(1)})$
with both $\pi(\EE_{(2,-1)}^{(1)})$ and $\pi(\EE_{(1,-1)}^{(1)})$ being nonzero, 
where $\EE_{(l,\eta)}^{(1)} = \EE \cup \{([2,-1]_\rho,l,\eta), ([2,-1]_\rho,l,-\eta)\}$.
However, since 
$D^{(1)}_{\rho|\cdot|^3} \circ D^{(2)}_{\rho|\cdot|^2} \circ D^{(1)}_{\rho|\cdot|^1}(\pi(\EE)) \not= 0$ but
$D^{(1)}_{\rho|\cdot|^3} \circ D^{(2)}_{\rho|\cdot|^2} \circ D^{(1)}_{\rho|\cdot|^1}
(\pi(\EE_{(2,-1)}^{(1)})) = 0$, 
by Proposition \ref{middle}, we conclude that
$\soc(u_\rho(4,4)|\cdot|^1 \rtimes \pi(\EE)) 
= \soc(u_\rho(2,4)|\cdot|^2 \rtimes \pi(\EE_{(1,-1)}^{(1)}))$
is irreducible. 
Also, we note that 
\[
D_{|\cdot|^3}^{(2)} \circ D_{|\cdot|^2}^{(3)} \circ D_{|\cdot|^1}^{(2)}
\left( \soc(u_\rho(2,4)|\cdot|^2 \rtimes \pi(\EE_{(1,-1)}^{(1)})) \right)
\not= 0.
\]

\item
When $s = -1$, we have 
\[
u_\rho(4,4)|\cdot|^{-1} \rtimes \pi(\EE) 
\hookrightarrow u_\rho(4,2)|\cdot|^{-2} \times u_\rho(4,2) \rtimes \pi(\EE). 
\]
As above, we see that $u_\rho(4,2) \rtimes \pi(\EE) = \pi(\EE_{(0,-1)}^{(-1)})$ is irreducible, 
where $\EE_{(0,-1)}^{(-1)} = \EE \cup \{([2,1]_\rho,0,-1), ([2,1]_\rho,0,1)\}$.
In particular, 
$\soc(u_\rho(4,4)|\cdot|^{-1} \rtimes \pi(\EE)) 
= \soc(u_\rho(4,2)|\cdot|^{-2} \rtimes \pi(\EE_{(0,-1)}^{(-1)}))$
is also irreducible. 
Note that 
\[
D_{|\cdot|^3}^{(2)} \circ D_{|\cdot|^2}^{(3)} \circ D_{|\cdot|^1}^{(2)}
\left( \soc(u_\rho(4,2)|\cdot|^{-2} \rtimes \pi(\EE_{(0,-1)}^{(-1)})) \right) 
= 0.
\]
Hence we have
\[
\soc(u_\rho(4,4)|\cdot|^{-1} \rtimes \pi(\EE))
\not\cong 
\soc(u_\rho(4,4)|\cdot|^{1} \rtimes \pi(\EE)), 
\]
which means that $u_\rho(4,4)|\cdot|^{1} \rtimes \pi(\EE)$ is reducible. 
\end{enumerate}
Since $s_0 \in \Z$ by Theorem \ref{irred-bad}, we conclude that $s_0 = 1$.
\end{ex}

\begin{ex}
Let $\psi = \rho \boxtimes (S_2 \boxtimes S_2 + S_5 \boxtimes S_3) \in \Psi_\gp(\Sp_{18}(F))$.
Then $\Pi_\psi = \{\pi(\EE_i) \;|\; 1 \leq i \leq 5\}$ with 
\begin{align*}
\EE_1 &= \{ ([1, 0], 1, 1), ([3, 1], 1, 1) \}, \\
\EE_2 &= \{ ([1, 0], 0, -1), ([3, 1], 0, 1) \}, \\
\EE_3 &= \{ ([1, 0], 1, 1), ([3, 1], 0, -1) \}, \\
\EE_4 &= \{ ([1, 0], 0, 1), ([3, 1], 1, -1) \}, \\
\EE_5 &= \{ ([1, 0], 0, -1), ([3, 1], 1, -1) \}. 
\end{align*}
Note that $\Sc_\psi$ has exactly two characters. 
By \cite[Theorem 3.5]{At}, we have $\pair{\cdot, \pi(\EE_i)}_{\psi} = \1 \iff i = 1,3$.  
Now, for $1 \leq i \leq 5$, 
let $s_i$ be the first reducibility point for $u_\rho(4,2)|\cdot|^s \rtimes \pi(\EE_i)$.
Note that $s_i \in \Z$ by Theorem \ref{irred-bad}. 
We compute $s_i$ for $1 \leq i \leq 5$.

\begin{enumerate}
\item
When $s = 0$, 
by Theorem \ref{s=0} together with \cite[Theorem 1.4]{At}, we have 
\begin{align*}
u_\rho(4,2) \rtimes \pi(\EE_1) \cong 
&\pi(\EE_1 \cup \{([2, 1], 0, 1), ([2, 1], 0, -1)\})
\\&\oplus \pi(\EE_1 \cup \{ ([2, 1], 1, 1), ([2, 1], 1, -1)\}), \\
u_\rho(4,2) \rtimes \pi(\EE_2) \cong 
&\pi(\EE_2 \cup \{([2, 1], 0, 1), ([2, 1], 0, -1)\}), \\
u_\rho(4,2) \rtimes \pi(\EE_3) \cong 
&\pi(\EE_3 \cup \{([2, 1], 0, -1), ([2, 1], 0, 1)\}), \\
u_\rho(4,2) \rtimes \pi(\EE_4) \cong 
&\pi(\EE_4 \cup \{([2, 1], 1, 1), ([2, 1], 1, -1)\})
\\&\oplus \pi(\EE_4 \cup \{([2, 1], 0, -1), ([2, 1], 0, 1)\}), \\
u_\rho(4,2) \rtimes \pi(\EE_5) \cong 
&\pi(\EE_5 \cup \{([2, 1], 1, 1), ([2, 1], 1, -1)\}).
\end{align*}
In particular, $u_\rho(4,2) \rtimes \pi(\EE_i)$ is reducible 
if and only if $i = 1,4$ so that $s_1 = s_4 = 0$. 

\item
When $s = \pm1$, 
by Propositions \ref{middle}, \ref{image} and \ref{s>>0}, we have
\begin{align*}
\soc(u_\rho(4,2)|\cdot|^1 \rtimes \pi(\EE_2)) 
&\cong 
L(\Delta_\rho[0,-1],\Delta_\rho[1,-3]; \pi(0^-,1^-,2^-,2^-,3^+)), \\
\soc(u_\rho(4,2)|\cdot|^{-1} \rtimes \pi(\EE_2)) 
&\cong 
L(\Delta_\rho[0,-3],\Delta_\rho[1,-2]; \pi(0^-,1^+,1^+,2^-,3^+)), \\
\soc(u_\rho(4,2)|\cdot|^1 \rtimes \pi(\EE_3)) 
&\cong 
L(\Delta_\rho[0,-3],\Delta_\rho[0,-1],\Delta_\rho[1,-2]; \pi(1^-,2^+,3^-)), \\
\soc(u_\rho(4,2)|\cdot|^{-1} \rtimes \pi(\EE_3)) 
&\cong 
L(\Delta_\rho[0,-3],\Delta_\rho[0,-1],\Delta_\rho[1,-2]; \pi(1^-,2^+,3^-)), \\
\soc(u_\rho(4,2)|\cdot|^1 \rtimes \pi(\EE_5)) 
&\cong 
L(\Delta_\rho[0,-3],\Delta_\rho[2,-3]; \pi(0^-,1^+,1^+,1^+,2^-)), \\
\soc(u_\rho(4,2)|\cdot|^{-1} \rtimes \pi(\EE_5)) 
&\cong 
L(\Delta_\rho[0,-3],\Delta_\rho[1,-3],\Delta_\rho[1,-2]; \pi(0^-,1^+,2^-)).
\end{align*}
In particular, for any $i = 2,3,5$, 
the socle $\soc(u_\rho(4,2)|\cdot|^{\pm1} \rtimes \pi(\EE_i))$ is irreducible. 
Since 
$\soc(u_\rho(4,2)|\cdot|^{1} \rtimes \pi(\EE_i)) \not\cong 
\soc(u_\rho(4,2)|\cdot|^{-1} \rtimes \pi(\EE_i))$ for $i = 2,5$, 
we have $s_2 = s_5 = 1$. 
On the other hand, $u_\rho(4,2)|\cdot|^{1} \rtimes \pi(\EE_3)$ is irreducible. 

\item
When $s = \pm2$, by Proposition \ref{s>>0}, we have 
\begin{align*}
\soc(u_\rho(4,2)|\cdot|^2 \rtimes \pi(\EE_3)) 
&\cong 
L(\Delta_\rho[0,-3],\Delta_\rho[1,-2]; \pi(1^-,3^+,4^-)), \\
\soc(u_\rho(4,2)|\cdot|^{-2} \rtimes \pi(\EE_3)) 
&\cong 
L(\Delta_\rho[-1,-4],\Delta_\rho[0,-3],\Delta_\rho[0,-1]; \pi(1^-,2^+,3^-)). 
\end{align*}
Hence $u_\rho(4,2)|\cdot|^{2} \rtimes \pi(\EE_3)$ is reducible so that $s_3 = 2$.
\end{enumerate}
\end{ex}

\appendix
%\section{Computations for certain derivatives}
%\section{Computations for certain derivatives}
\section{Computations for certain derivatives}\label{appA}
Recall that when $\pi \in \Irr(G_n)$ is $\rho|\cdot|^{-1}$-reduced (\resp $\rho|\cdot|^1$-reduced), 
the highest $\Delta_\rho[0,-1]$-derivative $D^{\max}_{\Delta_\rho[0,-1]}(\pi)$
(\resp the highest $Z_\rho[0,1]$-derivative $D^{\max}_{Z_\rho[0,1]}(\pi)$)
is irreducible (\cite[Proposition 3.7]{AM}).
In \cite{AM}, 
explicit formulas for $\Delta_\rho[0,-1]$-derivatives and for $Z_\rho[0,1]$-derivatives were 
given only for irreducible representations satisfying some specific conditions.
The goal of this appendix is to compute these derivatives for $\pi$ of good parity in general. 
\par

Here, we say that an irreducible representation $\pi$ is \emph{of good parity}
if $\pi$ is a subrepresentation of an induced representation of the form 
$\rho_1|\cdot|^{s_1} \times \dots \times \rho_r|\cdot|^{s_r} \rtimes \sigma$, 
where
\begin{itemize}
\item
$\rho_i \in \Cusp^\bot(\GL_{d_i}(F))$ and $s_i \in (1/2)\Z$; 
\item
$\sigma$ is an irreducible supercuspidal representation of $G_{n_0}$; 
\item
$\rho_i|\cdot|^{s_i + m_i} \rtimes \sigma$ is reducible for some $m_i \in \Z$.
\end{itemize}

%\subsection{Derivatives for $\GL_n(F)$}
\subsection{Derivatives for $\GL_n(F)$}\label{der.gl}
Before dealing with classical groups, 
we fix notations and recall some facts on representations of $\GL_n(F)$. 
For these facts, see \cite{LM} and its references.
\par

Denote $P_{(m,n-m)}$ by the maximal standard parabolic subgroup of $\GL_n(F)$
with Levi $\GL_{m}(F) \times \GL_{n-m}(F)$. 
For a smooth representation $\tau$ of $\GL_n(F)$ of finite length, 
define the \emph{left $\rho|\cdot|^x$-derivative} $L_{\rho|\cdot|^x}^{(k)}(\tau)$ and 
the \emph{right $\rho|\cdot|^x$-derivative} $R_{\rho|\cdot|^x}^{(k)}(\tau)$ by 
\begin{align*}
[\Jac_{P_{(dk,n-dk)}}(\tau)] 
&= (\rho|\cdot|^x)^k \boxtimes L_{\rho|\cdot|^x}^{(k)}(\tau) + (\text{others}), \\
[\Jac_{P_{(n-dk,dk)}}(\tau)] 
&= R_{\rho|\cdot|^x}^{(k)}(\tau) \boxtimes (\rho|\cdot|^x)^k + (\text{others}).
\end{align*}
The highest derivatives $L^{\max}_{\rho|\cdot|^x}(\tau)$ and $R^{\max}_{\rho|\cdot|^x}(\tau)$
are defined similar as in Section \ref{s.der}. 
It is known that if $\tau$ is irreducible, 
then $L^{\max}_{\rho|\cdot|^x}(\tau)$ and $R^{\max}_{\rho|\cdot|^x}(\tau)$ are also irreducible
(see \cite[Lemma 2.1]{LM}). 
Moreover, 
the Langlands data for $L^{\max}_{\rho|\cdot|^x}(\tau)$ (\resp $R^{\max}_{\rho|\cdot|^x}(\tau)$) 
can be described from those for $\tau$ explicitly, and vice versa
(see, e.g., \cite[Theorem 5.11]{LM}).
\par

Similarly, following \cite[Section 3.4]{AM}, we can define 
\begin{itemize}
\item
the \emph{highest left $\Delta_\rho[0,-1]$-derivative} $L^{\max}_{\Delta_\rho[0,-1]}(\tau)$; 
\item
the \emph{highest right $\Delta_\rho[0,-1]$-derivative} $R^{\max}_{\Delta_\rho[0,-1]}(\tau)$; 
\item
the \emph{highest left $Z_\rho[0,1]$-derivative} $L^{\max}_{Z_\rho[0,1]}(\tau)$; 
\item
the \emph{highest right $Z_\rho[0,1]$-derivative} $R^{\max}_{Z_\rho[0,1]}(\tau)$. 
\end{itemize}
If $\tau$ is irreducible and left $\rho|\cdot|^{1}$-reduced, 
i.e., if $L_{\rho|\cdot|^1}^{(1)}(\tau) = 0$,
then $L^{\max}_{Z_\rho[0,1]}(\tau)$ is also irreducible. 
In this case, 
if we write 
$L^{\max}_{\rho|\cdot|^{1}}\circ L^{\max}_{\rho}(\tau) 
= L^{(k_1)}_{\rho|\cdot|^{1}}\circ L^{(k_0)}_{\rho}(\tau) = \tau'$, 
then $k_0 \geq k_1$ and we have 
\[
L^{\max}_{Z_\rho[0,1]}(\tau) 
= L^{(k_1)}_{Z_\rho[0,1]}(\tau) 
= \soc \left( \rho^{k_0-k_1} \times \tau' \right).
\]
Similar properties hold for other derivatives. 
These facts can be proven by the same argument as \cite[Lemma 3.5]{AM}.
\par

On the other hand, for any irreducible representation $\tau$ of $\GL_n(F)$, 
the socle of $Z_\rho[0,1]^r \times \tau$ is irreducible, and it can be computed by
\[
\soc(Z_\rho[0,1]^r \times \tau) 
= \soc\left(
\rho^{k_0+r} \times
\soc\left( (\rho|\cdot|^1)^r \times L_\rho^{(k_0)}(\tau) \right)
\right), 
\]
where we write $L_\rho^{\max}(\tau) = L_\rho^{(k_0)}(\tau)$.
Similar properties hold for the socles of 
$\tau \times Z_\rho[0,1]^r$, $\Delta_{\rho}[0,-1]^r \times \tau$ and $\tau \times \Delta_\rho[0,-1]^r$.
See \cite[Proposition 5.6]{LM}.

%\subsection{$\Delta_\rho[0,-1]$-derivatives}
\subsection{$\Delta_\rho[0,-1]$-derivatives}\label{sec.[0-1]}
Let $\pi$ be an irreducible representation of $G_n$. 
Suppose that $\pi$ is $\rho|\cdot|^{-1}$-reduced.  
Then $D^{\max}_{\Delta_\rho[0,-1]}(\pi)$ is irreducible (\cite[Proposition 3.7]{AM}). 
In a special case, an explicit formula for $D^{\max}_{\Delta_\rho[0,-1]}(\pi)$ was 
given in \cite[Proposition 3.8]{AM}.
In this subsection, we generalize this formula. 
\par

\begin{prop}
Write $\pi = L(\Delta_{\rho_1}[x_1,y_1], \dots, \Delta_{\rho_r}[x_r,y_r]; \pi_\temp)$ 
as in the Langlands classification. 
Suppose that $\pi$ is $\rho|\cdot|^{-1}$-reduced.  
Then 
\[
D^{\max}_{\Delta_\rho[0,-1]}(\pi)
\hookrightarrow 
L^{\max}_{\Delta_\rho[0,-1]}\left(L(\Delta_{\rho_1}[x_1,y_1], \dots, \Delta_{\rho_r}[x_r,y_r])\right)
\rtimes \pi_\temp. 
\]
\end{prop}
\begin{proof}
Write $\tau = L(\Delta_{\rho_1}[x_1,y_1], \dots, \Delta_{\rho_r}[x_r,y_r])$ and 
$L^{\max}_{\Delta_\rho[0,-1]}(\tau) = L^{(k)}_{\Delta_\rho[0,-1]}(\tau)$.
Note that $L^{\max}_{\Delta_\rho[0,-1]}(\tau)$ is irreducible 
since $\tau$ is left $\rho|\cdot|^{-1}$-reduced. 
Clearly, we have an inclusion 
\[
\pi \hookrightarrow 
\Delta_\rho[0,-1]^k \times 
L^{(k)}_{\Delta_\rho[0,-1]}(\tau) \rtimes \pi_\temp. 
\]
Since
\begin{itemize}
\item
$L^{\max}_{\Delta_\rho[0,-1]}(\tau)$ is left $\rho|\cdot|^{-1}$-reduced; 
\item
$x_i+y_i < 0$ so that $y_i \not= 0, 1$; 
\item
$\pi_\temp$ is $\rho|\cdot|^{-1}$-reduced (Casselman's criterion), 
\end{itemize}
we see that $D^{(k)}_{\Delta_\rho[0,-1]}(\pi)$ is the highest $\Delta_\rho[0,-1]$-derivative, 
and 
\[
D^{(k)}_{\Delta_\rho[0,-1]}(\pi) \hookrightarrow 
L^{(k)}_{\Delta_\rho[0,-1]}(\tau) \rtimes \pi_\temp. 
\]
This completes the proof.
\end{proof}

%\subsection{$Z_\rho[0,1]$-derivatives: A special case}
\subsection{$Z_\rho[0,1]$-derivatives: A special case}
Let $\pi$ be an irreducible representation of $G_n$. 
Suppose that $\pi$ is of good parity and $\rho|\cdot|^1$-reduced.  
Then $D^{\max}_{Z_\rho[0,1]}(\pi)$ is irreducible (\cite[Proposition 3.7]{AM}). 
When $\pi$ is further $\rho|\cdot|^z$-reduced for any $z < 0$, 
an explicit formula for $D^{\max}_{Z_\rho[0,1]}(\pi)$ was 
given in \cite[Theorem 8.1, Proposition 8.4]{AM}.
In this and next subsections, we generalize this formula. 
\par

Here, we consider a special case, which is the main case. 
Suppose that $\pi$ is of the form
\[
\pi = L((\rho|\cdot|^{-1})^s, \Delta_\rho[0,-1]^t; \pi(\phi,\ep))
\]
for $s,t \geq 0$ and $\phi \in \Phi_\gp(G_{n_0})$.
Set 
\[
\delta = \left\{
\begin{aligned}
&1 \iif 
\text{$\rho, \rho \boxtimes S_3 \subset \phi$ and $\ep(\rho)\ep(\rho \boxtimes S_3) \not= (-1)^t$}, \\
&0 \other.
\end{aligned}
\right. 
\]
Then by \cite[Theorem 7.1]{AM}, 
we have $D_{\rho|\cdot|^1}^{\max}(\pi) = D_{\rho|\cdot|^1}^{(k)}(\pi)$ with 
\[
k = \min\{s-m_\phi(\rho)+\delta, 0\} + m_\phi(\rho \boxtimes S_3) - \delta, 
\]
where 
$m_\phi(\rho)$ (\resp $m_\phi(\rho \boxtimes S_3)$) denotes 
the multiplicity of $\rho$ (\resp $\rho \boxtimes S_3$) in $\phi$.
In particular, $\pi$ is $\rho|\cdot|^1$-reduced if and only if 
$m_\phi(\rho \boxtimes S_3) = \delta$ and $s \leq m_\phi(\rho) - \delta$. 
The following is a generalization of \cite[Proposition 8.4]{AM}. 

\begin{prop}
Let $\pi = L((\rho|\cdot|^{-1})^s, \Delta_\rho[0,-1]^t; \pi(\phi,\ep))$ be as above. 
Suppose that $\pi$ is $\rho|\cdot|^1$-reduced. 
Write $m = m_\phi(\rho)$ so that $s \leq m-\delta$. 

\begin{enumerate}
\item
If $\delta = 1$ and $m \equiv s+1 \bmod 2$, 
then the highest $Z_\rho[0,1]$-derivative of $\pi$ is 
\[
D_{Z_\rho[0,1]}^{(t)}(\pi)
= \left\{
\begin{aligned}
&L((\rho|\cdot|^{-1})^s; \pi(\phi, \ep)) \iif t \equiv 0 \bmod 2, \\
&L((\rho|\cdot|^{-1})^{s+1}; \pi(\phi + \rho - \rho \boxtimes S_3, \ep)) \iif t \equiv 1 \bmod 2.
\end{aligned}
\right. 
\]

\item
If $\delta = 1$ and $m \equiv s \bmod 2$, 
then the highest $Z_\rho[0,1]$-derivative of $\pi$ is 
\[
D_{Z_\rho[0,1]}^{(t+1)}(\pi)
= \left\{
\begin{aligned}
&\pi(\phi - \rho - \rho \boxtimes S_3, \ep') \iif t \equiv 0 \bmod 2, s = 0, \\
&L((\rho|\cdot|^{-1})^{s-1}; \pi(\phi - \rho^{2}, \ep)) \iif t \equiv 0 \bmod 2, s>0, \\
&L((\rho|\cdot|^{-1})^s; \pi(\phi - \rho - \rho \boxtimes S_3, \ep)) \iif t \equiv 1 \bmod 2, 
\end{aligned}
\right. 
\]
where $\ep'$ is given so that 
$\ep'(\rho' \boxtimes S_d) \not= \ep(\rho' \boxtimes S_d) 
\iff \rho' \boxtimes S_d = \rho \boxtimes S_1$. 

\item
If $\delta = 0$ and $m \equiv s+1 \bmod 2$, 
then the highest $Z_\rho[0,1]$-derivative of $\pi$ is 
\[
\left\{
\begin{aligned}
D_{Z_\rho[0,1]}^{(0)}(\pi) 
&= L((\rho|\cdot|^{-1})^s; \pi(\phi, \ep)) \iif t = 0, \\
D_{Z_\rho[0,1]}^{(t-1)}(\pi) 
&= L((\rho|\cdot|^{-1})^{s+1}; \pi(\phi + \rho^{2}, \ep)) \iif t > 0, t \equiv 0 \bmod 2, \\
D_{Z_\rho[0,1]}^{(t-1)}(\pi) 
&= L((\rho|\cdot|^{-1})^{s}, \Delta_\rho[0,-1]; \pi(\phi, \ep)) \iif t > 0, t \equiv 1 \bmod 2.
\end{aligned}
\right.
\]

\item
If $\delta = 0$ and $m \equiv s \bmod 2$, 
then the highest $Z_\rho[0,1]$-derivative of $\pi$ is 
\[
D_{Z_\rho[0,1]}^{(t)}(\pi)
= \left\{
\begin{aligned}
&\pi(\phi, \ep') \iif t \equiv 1 \bmod 2, m > s = 0, \\
&L((\rho|\cdot|^{-1})^{s-1}, \Delta_\rho[0,-1]; \pi(\phi - \rho^{2}, \ep)) 
\iif t \equiv 1 \bmod 2, m > s>0, \\
&L((\rho|\cdot|^{-1})^s; \pi(\phi, \ep)) 
\other, 
\end{aligned}
\right. 
\]
where $\ep'$ is the same as in (2).
\end{enumerate}
\end{prop}
\begin{proof}
The proof is essentially the same as \cite[Proposition 8.4]{AM}.
We only give a detail for the proof of (2). 
\par

Assume that $\delta = 1$ and $m \equiv s \bmod 2$.
Write $m = s+ 2u$ so that $u > 0$.
Note that $\pi \in \Pi_\psi$ with
\[
\psi = \phi - \rho^s + (\rho \boxtimes S_1 \boxtimes S_3)^s + (\rho \boxtimes S_2 \boxtimes S_2)^t.
\]
Since $\psi$ contains $\rho$ with multiplicity $2u$, 
by Theorem \ref{s=0}, we see that 
\begin{align*}
\pi 
&\hookrightarrow 
\rho^u \rtimes L((\rho|\cdot|^{-1})^s, \Delta_\rho[0,-1]^t; \pi(\phi - \rho^{2u}, \ep))
\\
&\hookrightarrow 
\rho^{u+t} \rtimes L((\rho|\cdot|^{-1})^{s+t}; \pi(\phi - \rho^{2u}, \ep)).
\end{align*}
Since $L((\rho|\cdot|^{-1})^{s+t}; \pi(\phi - \rho^{2u}, \ep)) 
= (\rho|\cdot|^{-1})^t \rtimes L((\rho|\cdot|^{-1})^{s}; \pi(\phi - \rho^{2u}, \ep))$
is irreducible, 
and since $L((\rho|\cdot|^{-1})^{s}; \pi(\phi - \rho^{2u}, \ep))$
belongs to $\Pi_{\psi_0}$ with $\psi_0 = \phi - \rho^{m} + (\rho \boxtimes S_1 \boxtimes S_3)^s$, 
by \cite[Proposition 8.3 (ii)]{X2}, we see that 
$D^{\max}_\rho(\pi) = L((\rho|\cdot|^{-1})^{s+t}; \pi(\phi - \rho^{2u}, \ep))$
up to a multiplicity.
\par

When $t$ is odd, since $\ep(\rho \boxtimes S_3) = \ep(\rho)$, 
we have 
\[
\pi \hookrightarrow 
\rho^{u+t} \times (\rho|\cdot|^{1})^{t+1} \rtimes 
L((\rho|\cdot|^{-1})^{s}; \pi(\phi - \rho^{2u-1} - \rho \boxtimes S_3, \ep)).
\]
Hence
\begin{align*}
\pi &\hookrightarrow Z_\rho[0,1]^{t+1} \times \rho^{u-1} \rtimes 
L((\rho|\cdot|^{-1})^{s}; \pi(\phi - \rho^{2u-1} - \rho \boxtimes S_3, \ep))
\\& \cong 
Z_\rho[0,1]^{t+1} \rtimes 
L((\rho|\cdot|^{-1})^{s}; \pi(\phi - \rho - \rho \boxtimes S_3, \ep)).
\end{align*}
On the other hand, when $t$ is even and $s > 0$, since $\ep(\rho \boxtimes S_3) \not= \ep(\rho)$,
we have 
\[
\pi \hookrightarrow 
\rho^{u+t} \times (\rho|\cdot|^{1})^{t+1} \rtimes 
L((\rho|\cdot|^{-1})^{s-1}; \pi(\phi - \rho^{2u}, \ep)).
\]
Hence
\begin{align*}
\pi &\hookrightarrow Z_\rho[0,1]^{t+1} \times \rho^{u-1} \rtimes 
L((\rho|\cdot|^{-1})^{s-1}; \pi(\phi - \rho^{2u}, \ep))
\\& \cong 
Z_\rho[0,1]^{t+1} \rtimes 
L((\rho|\cdot|^{-1})^{s-1}; \pi(\phi - \rho^2, \ep)).
\end{align*}
The last isomorphism follows from Theorem \ref{s=0}. 
The case where $s = 0$ was proven in \cite[Proposition 8.4]{AM}. 
Therefore, we obtain (2).
\end{proof}

The converse of this proposition is given as follows. 
\begin{cor}
Let $\pi = L((\rho|\cdot|^{-1})^s, \Delta_\rho[0,-1]^t; \pi(\phi,\ep))$ be as above. 
Suppose that $\pi$ is $\rho|\cdot|^1$-reduced.  
Write 
$D^{\max}_{Z_\rho[0,1]}(\pi) = D^{(k)}_{Z_\rho[0,1]}(\pi) 
= L((\rho|\cdot|^{-1})^{s'}, \Delta_\rho[0,-1]^{t'}; \pi(\phi',\ep'))$. 
Assume that $k > 0$. 
Set $m' = m_{\phi'}(\rho)$. 
\begin{enumerate}
\item
If $k$ is even and $t' = 1$, then 
\[
(s,t,\phi,\ep) = (s', k+1, \phi',\ep').
\]

\item
If $k$ is even, $t' = 0$ and $m' \equiv s' \bmod 2$, 
then 
\[
(s,t,\phi,\ep) = (s', k, \phi',\ep').
\]

\item
If $k$ is even, $t' = 0$, $m' \equiv s'+1 \bmod 2$ and $\phi' \supset \rho \boxtimes S_3$, 
then 
\[
(s,t,\phi,\ep) = (s',k,\phi',\ep'). 
\]

\item
If $k$ is even, $t' = 0$, $m' \equiv s'+1 \bmod 2$ and $\phi' \not\supset \rho \boxtimes S_3$, 
then $m' > 0$ and 
\[
(s,t,\phi,\ep) = (s',k-1,\phi'+\rho+\rho \boxtimes S_3,\ep)
\]
with $\ep(\rho) = \ep'(\rho)$ and $\ep(\rho \boxtimes S_3) = (-1)^k \ep(\rho)$. 

\item
If $k$ is odd and $t' = 1$,  
then $m' > 0$ and 
\[
(s,t,\phi,\ep) = (s'+1, k, \phi'+\rho^2,\ep)
\]
with $\ep(\rho) = \ep'(\rho)$.

\item
If $k$ is odd, $t' = 0$ and $m' = s'$, 
then 
\[
(s,t,\phi,\ep) = (s', k, \phi',\ep').
\]

\item
If $k$ is odd, $t' = 0$, $s' = 0 < m'$ and $m' \equiv 0 \bmod 2$, 
then
\[
(s,t,\phi,\ep) = (0, k, \phi',\ep)
\]
with $\ep(\rho) \not= \ep'(\rho)$.

\item
If $k$ is odd, $t' = 0$ $s' = 0 < m'$, $m' \equiv 1 \bmod 2$ and $\phi' \supset \rho \boxtimes S_3$, 
then $m' > 0$ and 
\[
(s,t,\phi,\ep) = (1, k-1 , \phi' + \rho^2,\ep').
\]

\item
If $k$ is odd, $t' = 0$ $s' = 0 < m'$, $m' \equiv 1 \bmod 2$ 
and $\phi' \not\supset \rho \boxtimes S_3$, 
then 
\[
(s,t,\phi,\ep) = (0, k-1, \phi' + \rho + \rho \boxtimes S_3,\ep)
\]
with $\ep(\rho) \not= \ep'(\rho)$ and $\ep(\rho \boxtimes S_3) = (-1)^k \ep(\rho)$. 

\item
If $k$ is odd, $t' = 0$, $0 < s' < m'$ and $m' \equiv s' \bmod 2$, 
then 
\[
(s,t,\phi,\ep) = (s'-1, k+1, \phi' - \rho^2,\ep'). 
\]

\item
If $k$ is odd, $t' = 0$, $0 < s' < m'$, 
$m' \equiv s'+1 \bmod 2$ and $\phi' \supset \rho \boxtimes S_3$, 
then 
\[
(s,t,\phi,\ep) = (s'+1, k-1, \phi' + \rho^2,\ep)
\]
with $\ep(\rho) = \ep'(\rho)$. 

\item
If $k$ is odd, $t' = 0$, $0 < s' < m'$, $m' \equiv s'+1 \bmod 2$ 
and $\phi' \not\supset \rho \boxtimes S_3$, 
then 
\[
(s,t,\phi,\ep) = (s'-1, k, \phi' - \rho + \rho \boxtimes S_3,\ep)
\]
with $\ep(\rho) = \ep'(\rho)$ and $\ep(\rho \boxtimes S_3) = (-1)^{k-1} \ep(\rho)$.  
\end{enumerate}
\end{cor}

%\subsection{$Z_\rho[0,1]$-derivatives: The general case}
\subsection{$Z_\rho[0,1]$-derivatives: The general case}
We continue to study $Z_\rho[0,1]$-derivatives.
Here, we consider the general case. 
The following is an algorithm to compute $D^{\max}_{Z_\rho[0,1]}(\pi)$, 
which is analogue to Jantzen's one \cite[Section 3.3]{J-dual}.
The proof is also similar and we omit it.

\begin{alg}\label{algD}
Let $\pi \in \Irr(G_n)$ be of good parity.
Assume that $\pi$ is $\rho|\cdot|^1$-reduced.
 
\begin{enumerate}
\item
We write
$\pi = L(\Delta_{\rho_1}[x_1,y_1], \dots, \Delta_{\rho_r}[x_r,y_r], 
(\rho|\cdot|^{-1})^s, \Delta_\rho[0,-1]^{t}; \pi(\phi,\ep))$ as in the Langlands classification, 
where
\begin{itemize}
\item
$\phi \in \Phi_\gp(G_{n_0})$; 
\item
$s,t \geq 0$; 
\item
$x_1+y_1 \leq \dots \leq x_r+y_r < 0$; 
\item
$\Delta_{\rho_i}[x_i,y_i] \not\cong \rho|\cdot|^{-1}, \Delta_\rho[0,-1]$ for $i = 1, \dots, r$.
\end{itemize}
Remark that if $\rho_i \cong \rho$ and $x_i+y_i = -1/2$, then $\rho|\cdot|^{-1} \in [x_i,y_i]_\rho$
so that 
$\Delta_{\rho_i}[x_i,y_i] \times \rho|\cdot|^{-1} \cong \rho|\cdot|^{-1} \times \Delta_{\rho_i}[x_i,y_i]$.
Note that $y_i \not= -1$ if $\rho_i \cong \rho$.

\item
Set 
\begin{align*}
\pi_A &= L((\rho|\cdot|^{-1})^s, \Delta_\rho[0,-1]^{t}; \pi(\phi,\ep)), \\
\pi_A' &= D^{\max}_{\rho|\cdot|^1}(\pi_A) = D^{(l_1)}_{\rho|\cdot|^1}(\pi_A), \\
\pi_A'' &= D^{\max}_{Z_\rho[0,1]}(\pi_A') = D^{(k_1)}_{Z_\rho[0,1]}(\pi_A').
\end{align*}
Note that $\pi_A'$ and $\pi_A''$ are of the same form as $\pi_A$.

\item
We have $\pi \hookrightarrow \tau \rtimes \pi_A''$, 
where 
\begin{align*}
\tau &= 
\soc(L(\Delta_{\rho_1}[x_1,y_1], \dots, \Delta_{\rho_r}[x_r,y_r], 
(\rho|\cdot|^{1})^{l_1}) \times Z_\rho[0,1]^{k_1})
\\&\cong 
L(\Delta_{\rho_1}[x_1,y_1], \dots, \Delta_{\rho_r}[x_r,y_r], \rho^{k_1}, (\rho|\cdot|^{1})^{k_1+l_1}).
\end{align*}

\item
Remark that $\tau$ is left $\rho|\cdot|^1$-reduced 
since $\pi$ is $\rho|\cdot|^1$-reduced.
Compute $\tau' = L_{Z_\rho[0,1]}^{\max}(\tau) = L_{Z_\rho[0,1]}^{(k)}(\tau)$.
It is of the form
\[
\tau' = 
L(\Delta_{\rho_1}[x'_1,y_1], \dots, \Delta_{\rho_r}[x'_r,y_r], \rho^{k_2}, (\rho|\cdot|^{1})^{k_2+l_2})
\]
with $l_2 \leq l_1$, $k_2 \leq k_1$ and $x'_1 + y_1 \leq \dots \leq x'_r + y_r < 0$.
Then 
\[
D^{\max}_{Z_\rho[0,1]}(\pi) \hookrightarrow \tau' \rtimes \pi_A''.
\]

\item
Compute 
\begin{align*}
\pi_B' &= \soc\left( Z_\rho[0,1]^{k_2} \rtimes \pi_A'' \right), \\
\pi_B &= \soc\left( (\rho|\cdot|^1)^{l_2} \rtimes \pi_B' \right).
\end{align*}
Then 
\[
D^{\max}_{Z_\rho[0,1]}(\pi) \hookrightarrow 
L(\Delta_{\rho_1}[x'_1,y_1], \dots, \Delta_{\rho_r}[x'_r,y_r]) \rtimes \pi_B. 
\]

\item
Note that $\pi_B$ is of the form 
$\pi_B = L((\rho|\cdot|^{-1})^{s'}, \Delta_\rho[0,-1]^{t'}; \pi(\phi',\ep'))$. 
We conclude that 
\[
D^{\max}_{Z_\rho[0,1]}(\pi) = 
L(\Delta_{\rho_1}[x'_1,y_1], \dots, \Delta_{\rho_r}[x'_r,y_r], 
(\rho|\cdot|^{-1})^{s'}, \Delta_\rho[0,-1]^{t'}; \pi(\phi',\ep')).
\]
\end{enumerate}
\end{alg}

Finally, we state an algorithm to compute $\soc(Z_\rho[0,1]^k \rtimes \pi)$.
\begin{alg}
Let $\pi \in \Irr(G_n)$ be of good parity.
Assume that $\pi$ is $\rho|\cdot|^1$-reduced.

\begin{enumerate}
\item
Write 
$\pi = L(\Delta_{\rho_1}[x_1,y_1], \dots, \Delta_{\rho_r}[x_r,y_r], 
(\rho|\cdot|^{-1})^s, \Delta_\rho[0,-1]^{t}; \pi(\phi,\ep))$ as in Algorithm \ref{algD} (1). 

\item
Let $\pi_A$, $\pi_A' = D^{(l_1)}_{\rho|\cdot|^1}(\pi_A)$, $\pi_A'' = D^{(k_1)}_{Z_\rho[0,1]}(\pi_A')$ 
be as in Algorithm \ref{algD} (2), and 
$\tau = L(\Delta_{\rho_1}[x_1,y_1], \dots, \Delta_{\rho_r}[x_r,y_r], 
\rho^{k_1}, (\rho|\cdot|^{1})^{k_1+l_1})$
be as in Algorithm \ref{algD} (3).

\item
Compute $\tau' = \soc(Z_\rho[0,1]^k \rtimes \tau)$. 
It is of the form
\[
\tau' = 
L(\Delta_{\rho_1}[x'_1,y_1], \dots, \Delta_{\rho_r}[x'_r,y_r], \rho^{k_2}, (\rho|\cdot|^{1})^{k_2+l_2})
\]
with $x'_1 + y_1 \leq \dots \leq x'_r + y_r < 0$.

\item
Compute 
\begin{align*}
\pi_B' &= \soc\left( Z_\rho[0,1]^{k_2} \rtimes \pi_A'' \right), \\
\pi_B &= \soc\left( (\rho|\cdot|^1)^{l_2} \rtimes \pi_B' \right).
\end{align*}
Then $\pi_B$ is of the form $\pi_B = L((\rho|\cdot|^{-1})^{s'}, \Delta_\rho[0,-1]^{t'}; \pi(\phi',\ep'))$. 
We conclude that 
\[
\soc(Z_\rho[0,1] ^k \rtimes \pi) = 
L(\Delta_{\rho_1}[x'_1,y_1], \dots, \Delta_{\rho_r}[x'_r,y_r], 
(\rho|\cdot|^{-1})^{s'}, \Delta_\rho[0,-1]^{t'}; \pi(\phi',\ep')).
\]
\end{enumerate}
\end{alg}

%References

\end{document}